\documentclass[10pt]{amsart}

\usepackage{amsfonts, amssymb}

\newtheorem{theorem}{Theorem}[section]

\newtheorem{corollary}[theorem]{Corollary}
\newtheorem{lemma}[theorem]{Lemma}
\newtheorem{proposition}[theorem]{Proposition}

\theoremstyle{definition}
\newtheorem{definition}[theorem]{Definition}
\theoremstyle{remark}
\newtheorem{remark}[theorem]{Remark}

\numberwithin{equation}{section}

\newcommand{\CC}{\mathbb C}
\newcommand{\RR}{\mathbb R}

\newcommand{\aut}{\mathrm{Aut\,}}
\newcommand{\dist}{\mathrm{ dist\, }}
\newcommand{\supp}{\mathrm{ supp\, }}
\def\NN{{\mathbb N}}

\def\PP{{\mathbb P}}

\def\cC{{\mathcal C}}
\def\cD{{\mathcal D}}

\def\cH{{\mathcal H}}

\def\cL{{\mathcal L}}

\def\cU{{\mathcal U}}
\def\cV{{\mathcal V}}

\def\dist{\mathrm{dist}\,}

\def\dsh{\mathrm{DSH}}
\def\var{\mathrm{var}}

\def\id{\mathrm{Id}}

\begin{document}

\title[Equidistribution of positive closed $(p,p)$-currents on $\PP^k$]{Local regularity of super-potentials and equidistribution of positive closed currents on $\PP^k$}%
%{Equidistribution of positive closed $(p,p)$-currents on $\PP^k$ with bounded super-potentials near an analytic subset for endomorphisms of $\PP^k$}%

\author{Taeyong Ahn}%
\address{(Ahn) KIAS, 85 Hoegiro, Dongdaemun-gu, Seoul 130-722, Republic of Korea
}%
%
%\address{() } 
\email{(Ahn) tahn@kias.re.kr}
\date{\today}
%\email{() }%
%\thanks{Research of the first and the third named authors is supported in part by the grant 2011-0030044 (The SRC-GAIA) of the NRF of Korea.}%
%\subjclass[2010]{32A10}%
%\keywords{}%

\begin{abstract}
In this paper, we introduce and study natural notions of local continuity/boundedness of super-potentials on $\PP^k$. Next, we prove an equidistribution theorem of positive closed $(p, p)$-currents of $\PP^k$, whose super-potentials are continuous/bounded near an invariant analytic subset, for holomorphic endomorphisms of $\PP^k$. We also consider the case of regular polynomial automorphisms of $\CC^k$.\\

%\noindent Titre: La r\'egularit\'e locale des super-potentiels et l'\'equidistribution des courants positifs ferm\'es dans $\PP^k$
%\medskip
%
%\noindent R\'esum\'e. Dans cet article, nous introduisons et \'etudions certains concepts concernant la r\'egularit\'e locale des super-potentiels. Nous prouvons ensuite un th\'eor\`eme d'\'equidistribution des $(p, p)$-courants positifs ferm\'es de $\PP^k$ dont les super-potentiels sont born\'ees pr\`es d'un sous-ensemble analytique invariant dans le cas des endomorphismes holomorphes de $\PP^k$. Nous consid\'erons aussi le cas des automorphismes polynomiaux r\'eguliers de $\CC^k$.
\end{abstract}
\maketitle

\section{Introduction}
Let $f$ be a holomorphic endomorphism of $\PP^k$ of algebraic degree $d \geq 2$ and $\omega$ the Fubini-Study form of $\PP^k$ such that $\int_{\PP^k}\omega^k=1$. It is well-known that $d^{-n}(f^n)^*\omega$ converges to a positive closed $(1, 1)$-current $T$ of unit mass on $\PP^k$ in the sense of currents. We call $T$ the Green current associated with $f$. The support of $T$ is the Julia set of $f$ and $T$ contains much information on the dynamics of $f$. For background, we refer the reader to \cite{sibony}.
\smallskip

The following equidistribution theorem of positive closed $(p, p)$-currents was proved in \cite{Ahn2015} and it extends a main result in \cite{DS2009} (in particular, we are interested in the case of $1<p<k$):
\begin{theorem}[See Theorem 1.3 in \cite{Ahn2015}]\label{thm:2015}
Let $f$ and $T$ be as above. Then, there exists a proper (possibly empty) invariant analytic subset $E$ for $f$ such that if $S\in\cC_p$ is smooth near $E$, then $d^{-pn}(f^n)^*S$ converges to $T^p$ exponentially fast in the sense of currents where $\cC_p$ is the set of positive closed $(p, p)$-currents of unit mass on $\PP^k$.
\end{theorem}

The purpose of this paper is to develop natural notions of local continuity and boundedness of super-potentials on $\PP^k$ (see Section \ref{sec:super-potentials}), and to improve Theorem \ref{thm:2015} by allowing $S\in\cC_p$ to have small singularities over $E$ as follows: %(see Corollary \ref{cor:contitobdd}; continuity implies boundedness)
\begin{theorem}\label{thm:mainthm}
Let $f$ and $T$ be as in Theorem \ref{thm:2015}. Then, there exists a proper (possibly empty) invariant analytic subset $E$ for $f$ such that if $S\in\cC_p$ is a current whose super-potential $\cU_S$ of mean $0$ is continuous/bounded near $E$ (or equivalently, $S\in\cC_p$ is PC/PB near $E$), then $d^{-pn}(f^n)^*S$ converges to $T^p$ exponentially fast in the sense of currents.
\end{theorem}

One well-known measurement of singularities of positive closed currents is the Lelong number. In the case of bidegree $(1, 1)$, the Lelong number works very well with their quasi-potentials. 
%A quasi-potential $u$ of a current $S\in\cC_1$ is a quasi-plurisubharmonic function (locally representable by a difference of a plurisubharmonic function and a smooth function) such that $S-\omega=dd^c u$ and $\sup_{\PP^k}u=0$. %Indeed, one useful fact is that the volume of the sublevel set $\{u<c\}$ can be estimated in terms of the Lelong number. 
On the other hand, even though super-potentials, which were introduced in \cite{DS2009} by Dinh-Sibony, can be regarded as an analogue of quasi-potentials for the currents in $\cC_p$ with $1<p<k$, it is not clear whether the Lelong number works well with super-potentials. %For applications, it is important for the notion of ``small'' singularities to be compatible with super-potentials.
%
%The theory of super-potentials was introduced in \cite{DS2009} by Dinh-Sibony. For a current $S_p\in\cC_p$ and for $m\in\RR$, the super-potential $\cU_{S_p}$ of mean $m$ of $S_p$ is a function defined as an action of a quasi-potential $U_{S_p}$ of $S_p$ of mean $m$ on the space $\cC_{k-p+1}$. In other words, we can formally write
%\begin{displaymath}
%\cU_{S_p}(R)=\langle U_{S_p}, R \rangle
%\end{displaymath}
%for $R\in\cC_{k-p+1}$ where $U_{S_p}$ is a current such that $S_p-\omega^p=dd^cU_{S_p}$ and $\langle U_S, \omega^{k-p}\rangle=m$. In the case of $p=1$, this notion coincides with the notion of the quasi-potential function.
So, we rather consider the singularities of quasi-potentials. If a quasi-potential function of $S\in\cC_1$ is continuous/bounded in an open subset $U\subseteq\PP^k$, the Lelong number of $S$ is $0$ in $U$. We generalize these to the continuity/boundedness of super-potentials in an open subset $U\subseteq\PP^k$ for currents $S\in\cC_p$. Indeed, Theorem \ref{thm:Lelong0} says that the Lelong number of $S$ is 0 in $U$. So, it is reasonable to regard them as ``small'' singularities. These notions extend the notions of continuous/bounded super-potentials in \cite{DS2009}.
\smallskip

The idea of the proof of Theorem \ref{thm:mainthm} is improving the proof of Theorem \ref{thm:2015} by adapting the localization of Green quasi-potentials used in 
%One of the essential points of the proof of Theorem \ref{thm:mainthm} is the computations of super-potentials in the neighborhoods of $E$. Here, we adapt the strategy in 
the proof of Proposition 2.3.6 in \cite{DS2009} with the symmetry of the Green quasi-potential kernel (see Section \ref{sec:super-potentials}) and some technical settings for the computations near $E$. In general, it is not easy to localize Green quasi-potentials since they are not closed. Instead, we estimate Green quasi-potentials from below by negative closed currents with small errors. The point is that by localizing negative closed currents, we obtain DSH currents (see Section \ref{sec:dsh,pc,pb}). The technical settings are to control the DSH norm of the localized currents and the small errors. The symmetry of the Green quasi-potential kernel links the localization of a current and the localization of its Green quasi-potential. We also have to handle a boundary integral type computation (Lemma \ref{lem:geotech}) for the localization, which is also different from the case of Theorem \ref{thm:2015}.
\smallskip

%\begin{remark}
%Theorem \ref{thm:mainthm} is a higher codimensional version of the following theorem, which can be easily deduced from equidistribution theorems in the case of $p=1$ such as the main theorem of \cite{Taflin2011}:
%\begin{theorem}
%Let $f$ and $T$ be as in Theorem \ref{thm:2015}. Then, there exists a proper (possibly empty) invariant analytic subset $E$ for $f$ such that if $S\in\cC_1$ is a current with its quasi-potential $u$ bounded in a neighborhood of $E$, then we have $d^{-n}(f^n)^*S$ converges to $T$ exponentially fast in the sense of currents.
%\end{theorem}
%\end{remark}

As a historical remark for this type of equidistribution for holomorphic endomorphisms of $\PP^k$, for the case of $p=1$, see, for instance, Forn{\ae}ss-Sibony (\cite{FS1995}), Russakovskii-Shiffman (\cite{RS}), Sibony (\cite{sibony}), Favre-Jonsson (\cite{brolinthm}, \cite{eigenvaluations}),  Guedj (\cite{Guedj}), Dinh-Sibony (\cite{DS2008}), Parra (\cite{Parra}) and Taflin (\cite{Taflin2011}), and for the case of $p=k$, see, for example, Freire-Lopes-Ma\~{n}\'{e} (\cite{FLM}), Lyubich (\cite{Lyubich}), Forn{\ae}ss-Sibony (\cite{FS1994}), Briend-Duval (\cite{BD}) and Dinh-Sibony (\cite{DS2010}). The case of $1<p<k$ has been studied in Dinh-Sibony (\cite{DS2009}) and Ahn (\cite{Ahn2015}). Theorem 5.4.4 in \cite{DS2009} corresponds to the case where $E$ is empty. Theorem \ref{thm:mainthm} implies Theorem 1.3 in \cite{Ahn2015}. 
\smallskip

We also consider equidistribution of positive closed currents for regular polynomial automorphisms of $\CC^k$. This equidistribution theorem was first observed in Remark 5.5.8 in \cite{DS2009}. Here, we give details of the proof. 

%For regular polynomial automorphisms, the set of critical points of $f$ as a birational map of $\PP^k$ is not a hypersurface, but quasi-potentials of $T_+$ are no longer H\"older continuous over the whole $\PP^k$ where $T_+$ is the Green current associated with $f$ (see Section \ref{sec:poly}).
\begin{theorem}[See Remark 5.5.8 in \cite{DS2009}]\label{thm:2ndmainthm}
Let $f:\CC^k\to\CC^k$ be a regular polynomial automorphism of $\CC^k$ of degree $d_+\geq 2$ and $p>0$ an integer such that $\dim I_+=k-p-1$ and $\dim I_-=p-1$ where $I_\pm$ are the indeterminancy sets of $f^{\pm 1}$, repsectively. Let $T_+$ be the Green current associated with $f$. Then, for $S\in\cC_p$ whose super-potential $\cU_S$ of mean $0$ is continuous/bounded near $I_-$ (or equivalently, $S\in\cC_p$ is PC/PB near $I_-$), $d_+^{-pn}(f^n)^*S$ converges to $T_+^p$ in the sense of currents.
\end{theorem}

Furthermore, we estimate the speed of convergence in some cases and improve Remark 11.2 in \cite{Ahn2015}. The set $K_+$ below denotes the set of points of bounded orbit under $f$.
\begin{theorem}\label{thm:2ndmainthmspeed}
Let $f$ be a regular polynomial automorphism of degree $d_+\geq 2$ and $1\leq s\leq p$. Let $W_1$ and $W_2$ be neighborhoods of $I_-$ such that $W_1\Subset W_2\Subset\PP^k\setminus \overline{K_+}$, and $\chi:\PP^k\to [0, 1]$ a smooth cut-off function with $\chi\equiv 1$ on $W_1$ and $\supp \chi\subset W_2$. Let $S\in\cC_s$ be such that $|\cU_S(dd^c(\chi U_R))|<c$ for every smooth $R\in\cC_{k-s+1}$ where $\cU_S$ is the super-potential of $S$ of mean $0$ and $U_R$ is the Green quasi-potential of $R$ for a fixed Green quasi-potential kernel. Then, we have
\begin{displaymath}
d_+^{-sn}(f^n)^*(S)\to T_+^s
\end{displaymath} 
exponentially fast on $\cD^{k-s}(W)$ in the sense of currents where $W$ is an open subset relatively compact in $\PP^k\setminus I_+$ and $\cD^{k-s}(W)$ is the set of smooth $(k-s, k-s)$-forms $\varphi$ with $\supp \varphi\Subset W$. In particular, due to the extension theorem of Harvey-Polking in \cite{HarveyPolking}, we have equidistribution for every such $S\in\cC_s$.
\end{theorem}

A crucial difference from the case of holomorphic endomorphisms of $\PP^k$ is that the behavior of $f^{\pm 1}$ is very wild near $I_\pm$. The idea for the proof of Theorem \ref{thm:2ndmainthm} is that instead of directly showing the convergence, we use the rigidity of the set $K_+$ as in Section 5.5 in \cite{DS2009}. Concerning the proof of Theorem \ref{thm:2ndmainthmspeed}, we basically use similar arguments to the proof of Theorem \ref{thm:mainthm}. The main difference is that since quasi-potentials of $f_*(\omega)$ are not H\"older continuous near $I_-$ and $f_*(R)$ is smooth outside $I_-$ for a smooth $R\in\cC_{k-s+1}$, we split $\PP^k$ into a neighborhood of $I_-$, which is invariant under $f$, and its complement. We apply the hypotheses to the neighborhood of $I_-$ and the smoothness of $f_*(R)$ to the complement set.

%For example, in the case of bidegree $(1, 1)$, the definitions are equivalent to the continuity/boundedness of quasi-potentials in an open subset of $\PP^k$, and %if a super-potential is continuous near an analytic subset $E$, then it is bounded near $E$, and 
%if a super-potential $\cU_S$ of $S\in\cC_p$ of mean $0$ is bounded in an open subset $U\subset \PP^k$, then $S$ does not charge any mass on any analytic subsets in $U$ and the Lelong number of $S$ should be $0$ in $U$. Thus, the local boundedness of super-potentials really implies small singularities. Moreover, we also define local versions of PC/PB currents (for PC/PB currents, see \cite{DS2005}, \cite{DS2006} and \cite{DS2009}) and %Dinh-Sibony proved that currents admitting continuous/bounded super-potentials are equivalent to PC/PB currents in \cite{DS2009}. 
%show the equivalence between currents PC/PB near an analytic subset of $\PP^k$ and currents admitting super-potentials continuous/bounded near an analytic subset of $\PP^k$.
\medskip
%There are several difficulties. In addition to those in \cite{Ahn2015}.
%difficulties:
%\begin{enumerate}
%\item existence of the set of critical values and necessity of super-potentials and $\theta$-regularization
%\item localizing the computations near $E$: cut-off function and handling the boundary: tech: estimates of Green quasi-potential kernel + geometric argument
%\item some local computations in a small neighborhood: good cut-off function lemma.
%\end{enumerate}

This paper is organized as follows. In Section 2, we discuss currents. In Section 3, we introduce currents PC/PB near an analytic subset of $\PP^k$. In Section 4, we discuss super-potentials and introduce super-potentials continuous/bounded near an analytic subset of $\PP^k$. In Section 5, we discuss multiplicities related to $f^n$. In Section 6, we briefly summarize some Lojasiewicz type inequalities and construct a family of good cut-off functions. In Section 7, we prove Theorem \ref{thm:mainthm}. In Section 8, we prove Theorem \ref{thm:2ndmainthm} and Theorem \ref{thm:2ndmainthmspeed}.

%==========
%
%\noindent *** Clarify the independence of inequalities with respect to $n, i$\\
%\noindent *** Every super-potential is taken to be of mean $0$.\\
%\noindent *** Suppress the paper to a shorter one.\\
%\noindent *** structure of the paper and references.\\

\subsection*{Notations}
In this paper, we fix a local holomorphic coordinate chart for the complex Lie group $\aut(\PP^k)$ in a neighborhood of $\id\in\aut(\PP^k)$ and a smooth probability measure which has compact support inside the chart satisfying the following. Let $\zeta$ denote the coordinate system over the chart and $\tau_\zeta$ its corresponding automorphism in $\aut(\PP^k)$. We can choose the chart and a norm $\|\cdot\|_A$ on the chart so that $\zeta=0$ at $\id\in\aut(\PP^k)$, $\{\|\zeta\|_A<2\}$ lies inside the chart and $\|\cdot\|_A$ is invariant under the involution $\tau\mapsto\tau^{-1}$. Let $\rho$ denote the fixed smooth probability measure. We can choose $\rho$ so that $\rho$ is radial and decreasing as $\|\zeta\|_A$ increases, and $\supp \rho \Subset \{\|\zeta\|_A<1\}$. In particular, $\rho$ is preserved under the involution $\tau\mapsto\tau^{-1}$.

Below is a list of some frequently used notations:
\begin{itemize}
\item $\Delta$: the unit disc in $\CC$;
%\item $\omega$: the Fubini-Study form;
\item $\dist(\cdot, \cdot), \dist_{\mathrm{euc}}(\cdot, \cdot)$: the distances with respect to the Fubini-Study metric and with respect to the Euclidean metric, respectively;
\item $\lesssim, \gtrsim$: $\leq, \geq$ upto a constant, respectively. The dependence of the constant will be specified if necessary;
\item $\|\cdot\|$: the Euclidean norm on a Euclidean space or the mass of a positive current (cf. Section \ref{sec:currents}) and its meaning will be clear from the context;
\item $\|\cdot\|_{\cL^p}, \|\cdot\|_\infty, \|\cdot\|_{\cC^\alpha} ($resp., $\|\cdot\|_{\cL^p(U)}, \|\cdot\|_{\infty, U}, \|\cdot\|_{\cC^\alpha(U)})$: the $\cL^p$-norm, the sup-norm and the $\cC^\alpha$-norm of a function or the sum of such norms of the coefficients of a form over $\PP^k$ (resp., over $U\subset\PP^k$) with respect to a fixed finite atlas of $\PP^k$, respectively;
\item $A_\epsilon:=\{x\in\PP^k: \dist(x, A)<\epsilon\}$ for $\epsilon>0$ and a subset $A\subset \PP^k$, and the meaning in other cases such as in $x_j$ will be clear from the context;
\item $\nu(x, R)$: the Lelong number of a positive closed $(1, 1)$-current $R$ at $x$;
\item $\Psi_n$: the set of the critical values of $f^n$.
\end{itemize}

\subsection*{Acknowledgement}
The author would like to thank John Erik Forn{\ae}ss and Viet Anh Nguyen for helpful discussions and suggestions. The author expresses his thanks to Ngaiming Mok for suggestions. The author also would like to thank the referee for careful reading and for suggestions. Referee's comments and
suggestions were helpful in improving and clarifying the arguments.

\section{Currents}\label{sec:currents}
For basics of currents in complex analysis, we refer the reader to Sibony (\cite{sibony}) and Demailly (\cite{Demailly}). Here, we list some definitions and properties.% that will be frequently used.

Let $S$ be a positive current of bidegree $(p, p)$ on $\PP^k$. The mass $\|S\|$ of $S$ is equivalent to
\begin{displaymath}
\int_{\PP^k}S\wedge \omega^{k-p}
\end{displaymath}
and therefore, in this paper, we rather define the mass $\|S\|$ of $S$ as above. If $S$ is negative, we define $\|S\|:=-\int_{\PP^k}S\wedge\omega^{k-p}$.

Let $\cC_p$ be the space of positive closed $(p, p)$-currents of unit mass on $\PP^k$. The space $\cC_p$ is compact in the weak topology. The weak topology on $\cC_p$ is induced by the distance $d_{\cC_p}$ below. (cf. \cite{DS2009}) If $S$ and $S'$ are currents in $\cC_p$, then
\begin{displaymath}
d_{\cC_p}:=\sup_{\|\Phi\|_{\cC^1}\leq 1}|\langle S-S', \Phi\rangle|
\end{displaymath}
where $\Phi$ is a smooth $(k-p, k-p)$-form on $\PP^k$ and $\|\Phi\|_{\cC^1}$ is the sum of the $\cC^1$-norms of the coefficients of $\Phi$ with respect to a fixed finite atlas of $\PP^k$.

%The pull-back of a positive closed $(p, p)$-current by a holomorphic endomorphism of $\PP^k$ is well-defined as discussed in \cite{DS2007} and \cite{DS2009}. The Green current $T$ associated with $f$ is defined by $T:=\lim_{n\to\infty}d^{-n}(f^n)^*\omega$ and the Green $(p, p)$-current associated with $f$ is defined by $T^p$ thanks to the work of Bedford and Taylor (\cite{BT}). The currents $T$ and $T^p$ have special properties: it is invariant in the sense that $T=d^{-1}f^*T$; it is extremal in the cone $\{S\in\cC_1:S=d^{-1}f^*S\}$; $T$ admits H\"older continuous quasi-potentials; $T^p$ admits H\"older continuous super-potentials with respect to $d_{\cC_p}$ (see Section \ref{sec:super-potentials}). 

Smooth forms are dense in $\cC_p$. (cf. \cite{DS2009}) Let $h_\theta:\{\|\zeta\|_A<1\}\to\{\|\zeta\|_A<1\}$ be an endomorphism defined by $h_\theta(\zeta):=\theta\zeta$ for $\theta\in \Delta$ where $\{\|\zeta\|_A<1\}$ is the coordinate chart of $\aut(\PP^k)$ mentioned in Introduction. Define $\rho_\theta:=(h_\theta)_*(\rho)$.

\begin{definition}[See Section 2 in \cite{DS2009}]
For any $(p, p)$-current $R$ on $\PP^k$, we define the $\theta$-regularization $R_\theta$ of $R$ by
\begin{displaymath}
R_\theta:=\int_{\aut(\PP^k)}(\tau_\zeta)_*Rd\rho_\theta(\zeta)=\int_{\aut(\PP^k)}(\tau_{\theta\zeta})_*Rd\rho(\zeta)=\int_{\aut(\PP^k)}(\tau_{\theta\zeta})^*Rd\rho(\zeta)
\end{displaymath}
where $\tau_\zeta$ is the automorphism in $\aut(\PP^k)$ whose coordinate is $\zeta$.
\end{definition}

\begin{proposition}[See Proposition 2.1.6 in \cite{DS2009}]\label{prop:regCurrents}
If $\theta\neq 0$, then $R_\theta$ is a smooth form which depends continuously on $R$. Moreover, for every $\alpha\geq 0$ there is a constant $c_\alpha$ independent of $R$ such that
\begin{displaymath}
\|R_\theta\|_{\cC^\alpha}\leq c_\alpha\|R\||\theta|^{-2k^2-4k-\alpha}.
\end{displaymath}
If $K$ is a compact subset in $\Delta\setminus\{0\}$, then there is a constant $c_{\alpha, K}>0$ such that for $\theta, \theta'\in K$,
\begin{displaymath}
\|R_\theta-R_{\theta'}\|_{\cC^\alpha}\leq c_{\alpha, K}\|R\||\theta-\theta'|.
\end{displaymath}
\end{proposition}

If $R$ is positive and closed, then so is $R_\theta$. If $|\theta|=|\theta'|$, then $R_\theta=R_{\theta'}$. The mass of $R_\theta$ is the same as the mass of $R$ (see Lemma 2.4.1 in \cite{DS2009}).

\begin{lemma}[See Lemma 2.1.8 in \cite{DS2009}]\label{lem:regulCurrents}
Let $K$ be a compact subset of $\aut(\PP^k)$. Let $W$ and $W_0$ be open sets in $\PP^k$ such that $\overline{W_0}\subset\tau(W)$ for every $\tau\in K$. If $R$ is of class $\cC^\alpha$ on $W$ with $\alpha\geq 0$, then $\tau_*(R)$ is of class $\cC^\alpha$ on $W_0$. Moreover, there is a constant $\tilde{c}>0$ such that for all $\tau, \tau'\in K$,
\begin{displaymath}
\|\tau_*(R)\|_{\cC^\alpha(W_0)}\leq \tilde{c}\|R\|_{\cC^\alpha(W)}
\end{displaymath}
and
\begin{displaymath}
\|\tau_*(R)-{\tau'}_*(R)\|_{\cC^\alpha(W_0)}\leq \tilde{c}\|R\|_{\cC^\alpha(W)}\dist_\aut(\tau, \tau')^{\min\{\alpha, 1\}},
\end{displaymath}
where $\dist_\aut$ is with respect to a fixed smooth metric on $\aut(\PP^k)$.
\end{lemma}

\section{DSH functions/currents and PC/PB currents}\label{sec:dsh,pc,pb}
\subsection{DSH functions} 
The class of DSH functions can be regarded as a generalization of the class of quasi-plurisubharmonic (q-psh for short) functions. (cf. Dinh-Sibony (\cite{DS2006}))
\begin{definition}[See \cite{DS2006}]
An integrable function $\varphi$ on $\PP^k$ is said to be DSH if it is equal to a difference of two q-psh functions outside a pluripolar set.
\end{definition}
Two DSH functions are identified if they are equal to each other outside a pluripolar set. The DSH functions on $\PP^k$ form a vector space over $\RR$ equipped with the following norm:
\begin{displaymath}
\|\varphi\|_\dsh:=\|\varphi\|_{\cL^1}+\inf\{\|T^+\|:dd^c\varphi=T^+-T^-, T^\pm\textrm{ are positive closed}\}.
\end{displaymath}

\subsection{DSH and PC/PB currents} For DSH and PC/PB currents, please refer to \cite{DS2009} and references therein. In this subsection, we first give their definitions and introduce locally PC/PB currents.

Below, we describe them slightly differently from those in \cite{DS2009} for computational convenience, but essentially, they are the same for our purpose. 
%For the mass of a positive current can be described in a nice way. (cf. Section \ref{sec:currents})
\begin{definition}
A real $(p, p)$-current $\varphi$ on $\PP^k$ is said to be DSH if there exist negative $(p, p)$-currents $\varphi_\pm$ and positive closed $(p+1, p+1)$-currents $\Omega_\pm$ on $\PP^k$ such that $\varphi=\varphi_+-\varphi_-$ and $dd^c\varphi=\Omega_+-\Omega_-$. We define the DSH-norm, denoted by the same notation as for DSH functions, by
\begin{displaymath}
\|\varphi\|_\dsh:=\inf\{\|\varphi_+\|+\|\varphi_-\|+\|\Omega_+\| \},
\end{displaymath}
where the infimum is taken over all $\varphi_\pm$ and $\Omega_+$ as above.
\end{definition}

We denote $\dsh^p:=\{\varphi: \varphi$ is a DSH $(p, p)$-current on $\PP^k$ such that $\|\varphi\|_\dsh\leq 1\}$. Let $W$ be a subset of $\PP^k$. Denote $\dsh^p(W):=\{\varphi\in\dsh^p: \supp \varphi\Subset W\}$. Here, notice that the supports of $\varphi_\pm$ and $\Omega_\pm$ are not necessarily in $W$ while $\supp \varphi\Subset W$. We define topology on $\dsh^p$ and $\dsh^p(W)$ simply to be the subspace topology from the space of currents on $\PP^k$. 
%Notice that $\dsh^p(W)$ is not compact unless $W$ is compact, but it is no harm in our study. One advantage from using the subspace topologies of the space of currents on $\dsh^p$ and $\dsh^p(W)$ is that since the DSH-norm bounds the mass norm (for this definition, see \cite{Federer}) up to a universal constant, both spaces are metrizable. 
%Hence, in the below, in order to check a current being PC or PC in an open subset of $\PP^k$, it suffices to investigate convergent sequences.

%Observe the following. From the compactness of $\cC_p$, it is not difficult to see that for any $\varphi\in\dsh_{\PP^k}^p$, we can find $\varphi_i, \Phi_{i, \pm}$ such that $\varphi=\varphi_1-\varphi_2$, $dd^c\varphi_i=\Phi_{i, +}-\Phi_{i, -}$ and $\|\varphi\|_\dsh=\|\varphi_1\|+\|\varphi_2\|+\|\Phi_{1, +}\|+\|\Phi_{2, -}\|$. Hence, again from the compactness of $\cC_p$, it is not difficult to see that if $\{\psi_j\}\subset \dsh_{\PP^k}^p$ is a Cauchy sequence in the space of currents, then its limit is also in $\dsh_{\PP^k}^p$, which means that $\dsh_{\PP^k}^p$ is compact.

\begin{definition}
A current $S\in\cC_p$ is said to be PC if it can be extended to a linear form on $\dsh^{k-p}$ which is continuous with respect to the topology on $\dsh^{k-p}$. We write $\langle S, \Phi \rangle$ for the value of this linear form at $\Phi\in\dsh^{k-p}$.

A current $S\in\cC_p$ is said to be PB if there is a constant $c_S>0$ such that
\begin{displaymath}
|\langle S, \varphi\rangle|\leq c_S\|\varphi\|_\dsh
\end{displaymath}
for smooth real forms $\varphi$ of bidegree $(k-p, k-p)$.
\end{definition}

Now, we define their local versions. Let $W$ be an open subset of $\PP^k$ and $E$ a proper analytic subset of $\PP^k$.
\begin{definition}
A current $S\in\cC_p$ is said to be PC in $W$ if $S$ induces a linear form on $\dsh^{k-p}(W)$ which is continuous with respect to the topology on $\dsh^{k-p}(W)$. A current $S\in\cC_p$ is said to be PC near $E$ if there exists an open neighborhood $W_E$ of $E$ such that $S$ is PC in $W_E$.
\end{definition}

\begin{definition}
A current $S\in\cC_p$ is said to be PB in $W$ if there exists a constant $C_{S, W}>0$ such that
\begin{displaymath}
|\langle S, \varphi \rangle|\leq C_{S, W}\|\varphi\|_\dsh
\end{displaymath}
for smooth real forms $\varphi$ of bidegree $(k-p, k-p)$ whose support is relatively compact in $W$. A current $S\in \cC_p$ is said to be PB near $E$ if there exists a neighborhood $W_E$ of $E$ such that $S$ is PB in $W_E$.
%$\varphi\in\dsh_{\PP^k}^{k-p}(U_E)$.
\end{definition}

\begin{proposition}\label{prop:PCPB}
If $S\in\cC_p$ is PC in $W$, then for any open subset $W'\Subset W$, it is PB in $W'$. In particular, if $E$ is an analytic subset of $\PP^k$ and a current $S\in \cC_p$ is PC near $E$, then it is PB near $E$.
\end{proposition}

The proof is similar to that of Proposition 2.2 in \cite{DS2004} and so we omit it.
%\begin{proof} Denote by $\cD_\dsh^{k-p}(W')$ the set of smooth $(k-p, k-p)$-forms $\varphi$ whose support is $\Subset W'$ and such that $\|\varphi\|_\dsh=1$. Suppose to the contrary that $\sup_{\varphi\in\cD_\dsh^{k-p}(W')}\langle S, \varphi\rangle$ is NOT bounded. Then, we can find a sequence $\{\varphi_n\}\subset \cD_\dsh^{k-p}(W')$ such that $c_n:=|\langle S, \varphi_n\rangle|$ diverges to $+\infty$ and $c_n\geq 1$. Then, ${c_n}^{-1}\varphi_n$ converges to $0$ in $\dsh^{k-p}(W)$, and $|\langle S, {c_n}^{-1}\varphi_n \rangle|$ converges to $1$. This is a contradiction. Hence, $S$ is PB in $W'$. 
%\end{proof}

\section{Super-potentials}\label{sec:super-potentials}

\subsection{Preliminaries} The theory of super-potentials can be thought of as an analogue of the pluripotential theory of positive closed $(1, 1)$-currents for positive closed $(p, p)$-currents with $1<p<k$ and was introduced in Dinh-Sibony (\cite{DS2009}). We list some relevant definitions and properties. For details, see Dinh-Sibony (\cite{DS2009}).

\begin{definition}[See Section 3 in \cite{DS2009}]
Let $S$ be a smooth form in $\cC_p$ and $m$ a fixed real number. Then, the super-potential $\cU_S$ of $S$ of mean $m$ is the function on $\cC_{k-p+1}$ defined by
\begin{displaymath}
\cU_S(R):=\langle S, U_R\rangle
\end{displaymath}
where $U_R$ is a quasi-potential of $R$ of mean $m$, that is, a $(k-p, k-p)$-current $U_R$ such that $dd^cU_R=R-\omega^{k-p+1}$ and $\langle U_R, \omega^p\rangle=m$.

For general $S\in\cC_p$, we define
\begin{displaymath}
\cU_S(R):=\lim_{\theta\to 0}\cU_{S_\theta}(R)
\end{displaymath}
where the subscript $\theta$ means the $\theta$-regularization of $S$.
\end{definition}

\begin{remark} The super-potential does not depend on the choice of the quasi-potential $U_R$ as long as it is of mean $m$. 
\end{remark}

%\begin{definition}
%Let $\{S_n\}_{n\geq 0}$ be a sequence in $\cC_p$ converging to a current $S\in\cC_p$. Let $\cU_{S_n}$ (resp. $\cU_S$) be the super-potential of $S_n$ (resp. $S$) of mean $m_n$ (resp. $m$) such that $m_n$ converges to $m$. We say that $S_n$ H-converges to $S$ if $\cU_{S_n}\geq \cU_S$ for every $n$.
%\end{definition}
%
%Note that if $S_n$ H-converges to $S$, then $\cU_S=\lim_{n\to\infty} \cU_{S_n}$ with the definitions as above.

%\begin{proposition}[See Section 3 in \cite{DS2009}]
%Assume that $R$ is smooth. Then, $\cU_S(R)=\langle U_S, R\rangle$ where the super-potential $\cU_S$ of $S\in\cC_p$ is of mean $m$ and $U_S$ is a quasi-potential of $S$ of mean $m$. It is independent of the choice of the quasi-potential $U_S$ provided that it is of mean $m$. Moreover, $\cU_S(R)=\lim_{\theta\to 0}\cU_S(R_\theta)=\lim_{\theta\to 0}\cU_{S_\theta}(R)$ where the subscript $\theta$ means the $\theta$-regularization of currents and all the super-potentials are of mean $m$.
%\end{proposition}

Among various quasi-potentials, there is a good one, named the Green quasi-potential, from the computational perspective. It is given by an integral formula. We will call this kernel the Green quasi-potential kernel. %Our work heavily depends on the properties of the Green quasi-potential kernel in the proposition below.
\begin{proposition}[See Proposition 2.3.2 in ~\cite{DS2009}]\label{prop:kernel}
Consider $X:=\PP^k\times\PP^k$ and $D$ the diagonal of $X$. Let $\Omega(z, \xi):=\sum_{j=0}^k\omega(z)^j\wedge\omega(\xi)^{k-j}$, where $(z, \xi)$ denotes the homogeneous coordinates of $\PP^k\times\PP^k$ with $z=[z_0: ...: z_k]$ and $\xi=[\xi_0: ...: \xi_k]$. Then, there is a negative $(k-1, k-1)$-form $K$ on $X$ smooth outside $D$ such that $dd^cK=[D]-\Omega$ which satisfies the following inequalities near $D$:
\begin{displaymath}
\|K(\cdot)\|_{pt}\lesssim -\dist(\cdot, D)^{2-2k}\log \dist(\cdot, D)\,\,\,\,\,\textrm{ and }\,\,\,\,\,\|\nabla K(\cdot)\|_{pt}'\lesssim \dist(\cdot, D)^{1-2k}.
\end{displaymath}
Moreover, there are a negative $\dsh$ function $\eta$ and a positive closed $(k-1, k-1)$-form $\Theta$ smooth outside $D$ such that $K\geq \eta\Theta$, $\|\Theta(\cdot)\|_{pt}\lesssim \dist(\cdot, D)^{2-2k}$ and $\eta-\log\dist(\cdot, D)$ is bounded near $D$.
\end{proposition}
Here, $\|K(\cdot)\|_{pt}$ is the sum $\sum_j |K_j|$ and  $\|\nabla K(\cdot)\|_{pt}'$ is the sum $\sum_j\|\nabla K_j\|$, where the $K_j$'s are the coefficients of $K$ for a fixed finite atlas of $X$. On $\PP^k\times\PP^k$, the distance, which is also denoted by ``$\dist$'', is measured with respect to the product metric of the standard Fubini-Study metric on each $\PP^k$.

%\begin{lem}[See Lemma 2.3.3 in ~\cite{DS2009}]\label{lem:233}
%There is a negative DSH function $\eta$ on $X$ smooth outside $D$ such that $\eta-\log \dist(\cdot, D)$ is bounded.
%\end{lem}

%The Green quasi-potential is defined as follows:
\begin{definition}[See Section 3 in \cite{DS2009}]
The Green quasi-potential $U$ of $R\in\cC_p$ is the $(p-1, p-1)$-current on $\PP^k$ defined by
\begin{displaymath}
U(z):=\int_{\xi\neq z}R(\xi)\wedge K(z, \xi).
\end{displaymath}
\end{definition}

\begin{remark}\label{rem:Greenqpotential}
%We call $U$ the Green quasi-potential of $R$. 
Note that $U$ depends on the choice of $K$.
% The mean $m$ of $U$ is bounded by a constant independent of $R$. 
Once $K$ is fixed, the mean (or, equivalently mass) $m$ of $U$ is bounded uniformly with respect to $R$. Note also that $U-m\omega^{p-1}$ is a quasi-potential of mean $0$ of $R$.
\end{remark}

One observation is that we can choose $K(z, \zeta)=K(\zeta, z)$ and we will use the following lemma.
%One observation is that $\widetilde{K}(z, \zeta):=K(\zeta, z)$ is still one of the kernels satisfying Proposition \ref{prop:kernel}, but may possibly be different from $K(z, \zeta)$. In our work, once fixed, the choice of the Green quasi-potential kernel does not affect much. Hence, we will write $K(z, \zeta)$ for $\widetilde{K}(z, \zeta)$ as well for simpler descriptions. More precisely, we will use the following lemma. It can be viewed as a version of the Fubini theorem.
\begin{lemma}[See Section 2.3 in \cite{DS2010-1}]\label{lem:doublecurrent}
Let $K$ be the Green quasi-potential kernel as above. Let $R$ be a $(p, p)$-current of order $0$ and $\varphi$ a smooth $(k-p+1, k-p+1)$-current on $\PP^k$. Then, we have
\begin{displaymath}
\left\langle \int_{\xi\neq z} R(\xi)\wedge K(z, \xi), \varphi(z) \right\rangle=\left\langle R(\xi), \int_{z\neq \xi} \varphi(z)\wedge K(z, \xi) \right\rangle.
\end{displaymath}
\end{lemma}

\begin{theorem}[See Theorem 2.3.1]\label{thm:Greenqpotential}
Let $R$ be a current in $\cC_p$. Then, its Green quasi-potential $U$ is negative, depends linearly on $R$ and satisfies that for every $r$ and $s$ with $1\leq r<k/(k-1)$ and $1\leq s<2k/(2k-1)$, one has
\begin{displaymath}
\|U\|_{\cL^r}\leq c_r\,\,\, \textrm{ and }\,\,\,\|dU\|_{\cL^s}\leq c_s
\end{displaymath}
for some positive constants $c_r$ and $c_s$ independent of $R$. Moreover, $U$ depends continuously on $R$ with respect to $\cL^r$-topology on $U$ and the weak topology on $R$.
\end{theorem}

%The super-potentials have some analogous properties to those of quasi-potentials and pluri-subharmonic functions. We list three of them here. 
%Positive closed currents are completely determined by super-potentials.
%\begin{proposition}[See Proposition 3.1.9 in ~\cite{DS2009}]\label{prop:319}
%Let $I$ be a compact subset in $\PP^k$ with $(2k-2p)$-dimensional Hausdorff measure 0. Let $S$ and $S'$ be currents in $\cC_p$, with super-potentials $\cU_S$ and $\cU_{S'}$. If $\cU_S=\cU_{S'}$ on smooth forms in $\cC_{k-p+1}$ with compact support in $\PP^k\setminus I$, then $S=S'$.
%\end{proposition}
%
%The following is about the compactness property of super-potentials.
%\begin{proposition}[See Proposition 3.2.6 in ~\cite{DS2009}]\label{prop:326}
%Let $\cU_{S_n}$ be a super-potential of a current $S_n$ in $\cC_p$. Assume that $\{\cU_{s_n}\}_{n\geq 0}$ is bounded from above and does not converge uniformly to $-\infty$. Then there is an increasing sequence $\{n_j\}_{j\geq 0}$ of integers such that $S_{n_j}$ converge to a current $S$ and $\cU_{S_{n_j}}$ converge on smooth forms in $\cC_{k-p+1}$ to a super-potential $\cU_S$ of $S$. Moreover,
%\begin{displaymath}
%\limsup_{j\to\infty}\cU_{s_{n_j}}\leq\cU_S.
%\end{displaymath}
%\end{proposition}

The following theorem is about the regularity of the Green super-potentials of order $p$, that is, the super-potentials of the Green $(p,p)$-current $T^p$.
\begin{theorem}[See Theorem 5.4.1 in ~\cite{DS2009}]\label{thm:541}
Let $f:\PP^k\to\PP^k$ be a holomorphic map of algebraic degree $d\geq 2$. Then, the Green super-potentials of order $p$ of $f$ are H\"{o}lder continuous on $\cC_{k-p+1}$ with respect to the distance $d_{\cC_{k-p+1}}$ which induces the weak topology on $\cC_{k-p+1}$. 
\end{theorem}

\subsection{Continuous/Bounded super-potentials near an analytic subset}
%In this subsection, we extend in a natural way the definition of the super-potentials to more general currents first and consider super-potentials continuous/bounded in an open subset of $\PP^k$.

%Let $S\in\cC_p$ be a current and $\cU_S$ its super-potential of mean $0$. Notice that if $R_1, R_2\in\cC_{k-p+1}$ and one of $\cU_S(R_1)$ and $\cU_S(R_2)$ is finite, then the following can be naturally well-defined
%\begin{displaymath}
%\cU_S(r_1R_1+r_2R_2)=r_1\cU_S(R_1)+r_2\cU_S(R_2),
%\end{displaymath}
%for $r_1, r_2>0$.
%Define
%\begin{align*}
%&\cR_{k-p+1}:=\{rR: r>0, R\in\cC_{k-p+1}\}\quad\textrm{ and }\\
%&\cE_S:=\{rR: r>0, R\in\cC_{k-p+1} \textrm{ such that }\cU_S(R)=-\infty\}.
%\end{align*}
%We define an equivalence relation $\sim$ on $(\cR_{k-p+1}\times\cR_{k-p+1})\setminus (\cE_S\times\cE_S)$ by $(A, B)\sim(A', B')$ if $A-B=A'-B'$ in the sense of currents. Define $\cZ_S$ to be the set of equivalence classes $[(A,B)]$. Note that $\cC_s$ can be identified as a proper subset of $\cZ_S$. Each equivalence class is of the form $[(A,B)]=\{(A+C, B+C): C\in\cR_{k-p+1}\setminus \cE_S\}$. By abuse of notation, we write $A-B$ for the equivalence class $[(A,B)]$. Define
%\begin{displaymath}
%\cU_S(A-B):=\cU_S(A)-\cU_S(B) \textrm{ where }(A, B)\in\cZ_S.
%\end{displaymath}
%This map is well-defined on $\cZ_S$. 

%We first introduce notions about local regularity of super-potentials and next explore their properties. %Indeed, the definitions are actually paraphrases of the definitions of a current being PC/PB near $E$ in terms of the language of super-potentials.
We introduce notions about local regularity of super-potentials. For the norm $\|\cdot\|_\var$ below and related topology, see \cite{DS2010-1}.
\begin{definition}\label{def:bounded}
Let $S\in\cC_p$ and $W$ an open subset of $\PP^k$. The super-potential $\cU_S$ of $S$ of mean $m$ is said to be bounded in $W$ if there exists a constant $C_S>0$ such that for any smooth positive closed $(k-p+1, k-p+1)$-currents $R_+$ and $R_-$ of the same mass on $\PP^k$ with $\supp (R_+-R_-)\Subset W$, we have
\begin{displaymath}
|\cU_S(R_+-R_-)|=|\cU_S(R_+)-\cU_S(R_-)|\leq C_S\|R_+\|.
\end{displaymath}
Let $E$ be an analytic subset of $\PP^k$. We say that $\cU_S$ is bounded near $E$ if there exists a neighborhood $W_E$ of $E$ such that $\cU_S$ is bounded in $W_E$. 
\end{definition}

\begin{remark}
The bound $C_S$ in Definition \ref{def:bounded} is independent of the mean of $\cU_S$. Observe that the supports of $R_\pm$ are not necessarily in $W$.
\end{remark}

Let $1\leq l\leq k$ and $W$ an open subset of $\PP^k$. Let $\mathfrak{D}_\var^l(W)$ be the space of closed $(l, l)$-currents $R$ on $\PP^k$ such that $R$ can be written as $R=R_+-R_-$ and $\supp R\Subset W$ where $R_\pm$ are positive closed $(l, l)$-currents on $\PP^k$ of the same mass. Here, we do not require the supports of $R_\pm$ to be in $W$. Similarly to $\dsh^l$ and $\dsh^l(W)$, we can define a norm $\|R\|_\var=\inf\{\|R_+\|:R=R_+-R_-\}$ on $\mathfrak{D}_\var^l(W)$. 

Let $\mathfrak{D}^l(W):=\{R\in \mathfrak{D}_\var^l(W): \|R\|_\var\leq 1\}$. The topology on this space is the subspace topology of the space of currents. The norm $\|\cdot\|_\var$ bounds the mass norm. So, $\mathfrak{D}^l(W)$ is metrizable.

\begin{definition}\label{def:conti}
Let $S\in\cC_p$ and $W$ an open subset of $\PP^k$. The super-potential $\cU_S$ of $S$ of mean $m$ is said to be continuous in $W$ if $\cU_S$ induces a linear function on $\mathfrak{D}^{k-p+1}(W)$ continuous with respect to the topology of $\mathfrak{D}^{k-p+1}(W)$. Let $E$ be an analytic subset of $\PP^k$. We say that $\cU_S$ is continuous near $E$ if there exists a neighborhood $W_E$ of $E$ such that $\cU_S$ is continuous in $W_E$.
\end{definition}

\begin{remark}
When $W=\PP^k$, Definition \ref{def:bounded} and Definition \ref{def:conti} coincide with the notions of bounded/continuous super-potentials in \cite{DS2009}. %Indeed, we simply consider $R-\|R\|\omega^{k-p+1}$ for $R\in\cC_{k-p+1}$.
\end{remark}

\begin{remark}
From the definition, it is not difficult to see that if $S\leq S'$ in the sense of currents and if $\cU_{S'}$ is bounded near $E$, then so is $\cU_S$. Also, since super-potentials are upper-semicontinuous (usc for short) on $\PP^k$, the same is true for continuity.
\end{remark}

Let $E$ be an analytic subset of $\PP^k$. We first compare currents being PC/PB near $E$ and currents admitting super-potentials continuous/bounded near $E$.

\begin{proposition}\label{prop:equivalencePB}
Let $S\in\cC_p$ and $E$ an analytic subset of $\PP^k$. The super-potential $\cU_S$ of $S$ of mean $m$ is bounded near $E$ if and only if $S$ is PB near $E$. 
\end{proposition}

\begin{proof} We first suppose that the super-potential $\cU_S$ of $S$ of mean $m$ is bounded near $E$. Let $W_E$ denote an open neighborhood of $E$ where $\cU_S$ is bounded. Let $\varphi$ be a smooth real $(k-p, k-p)$-form with $\supp \varphi \Subset W_E$. Then, we have
\begin{align*}
|\langle S, \varphi\rangle|&\leq |\langle S-\omega^p, \varphi\rangle|+|\langle \omega^p, \varphi\rangle|\leq |\langle dd^cU_S, \varphi\rangle|+|\langle \omega^p, \varphi\rangle|\\
&\leq |\cU_S(dd^c\varphi)|+\|\varphi\|_\dsh\leq (1+c)\|\varphi\|_\dsh
\end{align*}
for some constant $c>0$ independent of $\varphi$. The last line is due to the boundedness of $\cU_S$ in $W_E$.

We prove the converse. Let $W_1$, $W_2$ and $W_3$ be the $\epsilon$-neighborhood, the $2\epsilon$-neighborhood and the $3\epsilon$-neighborhood of $E$, respectively such that $S$ is PB in $W_3$. Fix a cut-off function $\chi:\PP^k\to[0,1]$ such that $\supp \chi\Subset W_3$ and $\chi\equiv 1$ on $W_2$. Let $R:=R_+-R_-$ be a current such that $\supp R\Subset W_1$ and $R_\pm$ are smooth positive closed $(k-p+1, k-p+1)$-currents of the same mass on $\PP^k$.

We claim that $\psi_R:=\chi (U_{R_+}-U_{R_-})$ is DSH and $\|\psi_R\|_\dsh$ is uniformly bounded independently of $R$ where $U_{R_\pm}$ are the Green quasi-potentials of $R_\pm$, respectively. Observe that $\supp \psi_R\Subset W_3$ and that since $R_\pm$ are smooth, $U_R:=\int K\wedge R$ is well-defined and equals $U_{R_+}-U_{R_-}$. Let $\chi_R:\PP^k\to[0, 1]$ is a smooth cut-off function such that $\chi_R\equiv 1$ on $\supp R$ and $\supp \chi_R\Subset W_1$. Then, for any $z\in W_3\setminus W_2$, we have
\begin{align*}
U_R(z)&=\int_{\xi\neq z}K(z, \xi)\wedge R(\xi)=\int_{\xi\neq z}K(z, \xi)\wedge \chi_R(\xi)R(\xi)\\
&=\int_{\xi\neq z}K(z, \xi)\wedge \chi_R(\xi)R_+(\xi)-\int_{\xi\neq z}K(z, \xi)\wedge \chi_R(\xi)R_-(\xi).
\end{align*}
Denote by
\begin{displaymath}
U'_{R_\pm}(z):=\int_{\xi\neq z}K(z, \xi)\wedge \chi_R(\xi)R_\pm(\xi).
\end{displaymath}
Then, $\psi_R=\chi(U'_{R_+}-U'_{R_-})$, $\chi U'_{R_\pm}$ are negative and since $\dist(W_3\setminus W_2, W_1)\geq \epsilon$, Proposition \ref{prop:kernel} implies that $\|U'_{R_\pm}\|_{\infty, W_3\setminus W_2}$ is bounded by $c\|R_+\|$ where $c>0$ is a constant independent of $R$ and $R_\pm$. In the same way, $\|U'_{R_\pm}\|_{\cC^1, W_3\setminus W_2}$ is bounded by $c'\|R_+\|$ where $c'>0$ is a constant independent of $R$ and $R_\pm$. So, since
\begin{align*}
dd^c \psi_R= \chi (R_+-R_-) +dd^c\chi\wedge U_R-d^c\chi\wedge dU_R+d\chi \wedge d^cU_R,
\end{align*}
we know that $dd^c (\chi U_R) + R_-+M\|R_+\|\omega^{k-p+1}$ is positive closed for some $M>0$ independent of $R$ and $R_\pm$. Hence, the current $\psi_R$ is DSH as we have Remark \ref{rem:Greenqpotential} and its DSH norm is bounded by $c_S\|R\|_\var$ for some $c_S>0$ independent of $R$. Our hypothesis of $S$ being PB in $W_3$ implies that $|\langle S, \chi U_R\rangle|\leq c'_S\|R\|_\var$ for some $c'_S>0$.

We consider $\langle S, (1-\chi) U_R\rangle$. By the same argument as above, we can show that $(1-\chi)U_R$ is smooth and its sup-norm is bounded by $c''_S\|R\|_\var$ where $c''_S>0$ is a constant independent of $R$. Hence, we have
\begin{displaymath}
-c'''_S\|R\|_\var\leq \langle S, \chi U_R\rangle +\langle S, (1-\chi)U_R\rangle\leq \langle S, U_R\rangle=\cU_S(R),
\end{displaymath}
as desired for some $c'''_S>0$ independent of $R$.
\end{proof}

Similarly, we can prove the following and together with Proposition \ref{prop:PCPB}, we obtain a corollary below:
\begin{proposition}\label{prop:PCPBcontibdd}
Let $S\in\cC_p$ and $E$ an analytic subset of $\PP^k$. The super-potential $\cU_S$ of $S$ of mean $m$ is continuous near $E$ if and only if $S$ is PC near $E$. 
\end{proposition}

\begin{corollary}\label{cor:contitobdd}
If $\cU_S$ is continuous near $E$, then it is bounded near $E$.
\end{corollary}

Let $\cC_l(W)$ denote the set of currents in $\cC_l$ whose support is relatively compact in $W$ for an open subset $W\subseteq \PP^k$.
It is natural to expect the following.
\begin{proposition}\label{prop:PBforpositiveclosed}
Let $S\in\cC_p$ and $\cU_S$ a super-potential of $S$ of mean $m$. If $\cU_S$ is bounded in $W$, then for any open subset $\widetilde{W}\Subset W$, $\cU_S$ is bounded on $\cC_{k-p+1}(\widetilde{W})$.
\end{proposition}

\begin{proof}
Let $W'$ be an open subset such that $\widetilde{W}\Subset W'\Subset W$. Let $\chi:\PP^k\to[0,1]$ be a smooth function such that $\supp \chi \Subset W$ and $\chi\equiv 1$ on $W'$. Let $R$ be a current in $\cC_{k-p+1}(\widetilde{W})$ and $U_R$ the Green quasi-potential of $R$. It is not difficult to see that $\chi U_R$ is DSH and $\|\chi U_R\|_\dsh$ is bounded for $R\in\cC_{k-p+1}(\widetilde{W})$. By the standard regularization argument for currents on $\PP^k$, it is not difficult to see that $\cU_S(dd^c(\chi U_R))$ is bounded independently of $R$ from our hypothesis. We also have
\begin{align*}
\cU_S(dd^c(\chi U_R))=\lim_{\theta\to 0}\cU_{S_\theta}(dd^c(\chi U_R))=\lim_{\theta\to 0}\langle S_\theta-\omega^p, \chi U_R \rangle.
\end{align*}
For sufficiently large $M$ independent of $\theta$ and $R$, we have
\begin{align*}
\langle S_\theta, \chi U_R \rangle-\langle S_\theta, M\omega^{k-p}\rangle\leq \langle S_\theta, U_R\rangle.
\end{align*}
Hence, we have
\begin{align*}
\cU_S(R)&=\lim_{\theta\to 0}\langle S_\theta, U_R \rangle\geq\lim_{\theta\to 0}(\langle S_\theta, \chi U_R \rangle-\langle S_\theta, M\omega^{k-p}\rangle)\\
&\geq\cU_S(dd^c(\chi U_R))+\langle \omega^p, \chi U_R\rangle-\langle S, M\omega^{k-p}\rangle.
\end{align*}
Since the mass of the Green quasi-potential is uniformly bounded and super-potentials are bounded above, the proposition is proved.
\end{proof}

%\begin{remark}\label{rmk:W} If there exists an $R_0\in\cC_{k-p+1}(W)$ such that $\cU_S(R_0)$ is finite, then $\cU_S$ is bounded on $\cC_{k-p+1}(W)$. Also, the proof becomes simple. Indeed, just consider $R-R_0$ for $R\in\cC_{k-p+1}$.
%\end{remark}

\begin{corollary}
When $p=1$, Definition \ref{def:bounded} is equivalent to the local boundedness of quasi-potential \emph{functions} of the positive closed $(1, 1)$-current $S$ in the following sense: if a quasi-potential $u$ of $S$ is bounded in an open subset $W$ of $\PP^k$, then the super-potential $\cU_S$ of mean $m$ is bounded in $W$ and conversely, if $\cU_S$ is bounded in $W$, then for any open subset $\widetilde{W}\Subset W$, $u$ is bounded in $\widetilde{W}$.
\end{corollary}

\begin{proof}
The first statement is straightforward from a similar argument to Proposition \ref{prop:equivalencePB}. The converse is obtained from the previous proposition and the fact that the Dirac mass is also a positive closed $(k, k)$-current on $\PP^k$.
\end{proof}

\begin{proposition}\label{prop:nomass}
Let $S$ be a current whose super-potential of mean $m$ is bounded in an open subset $W$ in $\PP^k$ and $E\Subset W$ an analytic subset of $\PP^k$. Then $S$ has no mass on $E$.
\end{proposition}
The proof is similar to that of Proposition 3.3.5 in \cite{DS2009} and so we omit it.
%The proof mimics the proof of Proposition 3.3.5 in \cite{DS2009}.
%\begin{proof}
%We can find a strictly negative quasi-plurisubharmonic (qpsh for short) function $u:\PP^k\to\RR$ such that $E=\{u=-\infty\}$ and $u$ is smooth outside $E$. We can find a sequence $\{u_n\}_{n\geq 0}$ of negative smooth qpsh functions that decreasingly converges to $u$.
%
%Fix a smooth cut-off function $\chi:\PP^k\to[0, 1]$ such that $\chi\equiv 1$ near $E$ and $\supp \chi\Subset W$. Note that this sequence $\{u_n\}$ can be obtained from the $\theta$-regularization of $u$ and therefore, the $\cC_2$-norm of $u_n$ over $\supp (d\chi)\cup \supp (d^c\chi)\cup \supp (dd^c\chi)$ is uniformly bounded. Then, $\Omega_n:=dd^c(\chi u_n)\wedge \omega^{k-p}$ is a closed $(k-p+1, k-p+1)$-current such that $\|\Omega_n\|_\dsh$ is bounded independently of $n$ and with $\supp \Omega_n\Subset W$. Hence, we have
%\begin{align*}
%\cU_S(\Omega_n)=\langle U_S, dd^c(\chi u_n\omega^{k-p}) \rangle=\langle S-\omega^p, \chi u_n\omega^{k-p} \rangle.
%\end{align*}
%Since $\int \chi u_n\omega^k$ is bounded independently of $n$, the boundedness of $\cU_S$ in $W$ implies that $\langle S, \chi u_n\omega^{k-p} \rangle$ is bounded independently of $n$.
%As $n\to\infty$, $u_n$ decreases to $u$ and therefore, we have $\langle S, \chi u_n\omega^{k-p} \rangle$ decreasingly converging to $\langle S, \chi u\omega^{k-p} \rangle$.  This implies that $\langle S, \chi u\omega^{k-p} \rangle$ is bounded and that $S$ has no mass on $E$.
%\end{proof}
By the same kind of argument, we can prove the following theorem.
\begin{theorem}\label{thm:Lelong0}
Let $S\in\cC_p$. If its super-potential $\cU_S$ of mean $m$ is bounded in an open subset $U\subseteq\PP^k$, then the Lelong number of $S$ is zero in $U$.
\end{theorem}

\begin{proof} Let $z_1\in U$. Let $q(z)=\log\|z-z_1\|-\log\|z\|+c$ where $z$ is the homogeneous coordinate of $\PP^k$. By taking an appropriate constant $c$, we may assume that $q(z)<0$. Note that $q(z)$ only has $-\infty$ at $z_1$. Let $\omega_q:=\omega +dd^cq$. Here, we use the definition of the generalized Lelong number introduced by Demailly. Then, it suffices to prove
\begin{displaymath}
\int_{(z_1)_\epsilon} qS\wedge \omega_q^{k-p}>-\infty
\end{displaymath}
where $(z_1)_\epsilon$ is the $\epsilon$-neighborhood of $z_1$ for some $\epsilon>0$. Notice that $q$ is smooth except $z_1$. Based on the work of Demailly, $\omega_q^{k-p}$ is a positive closed current. Let $\chi:\PP^k\to [0, 1]$ be a cut-off function such that $\chi\equiv 1$ near $z_1$ and $\supp \chi\Subset U$. Let $q_\theta$ be the $\theta$-regularization of $q$. By adding constants, we may assume that $q_\theta$ decreasingly converges to $q$ near $z_1$ as $|\theta|\to 0$. For all sufficiently large $n$, $\chi q_{1/n}$ is a sequence of negative smooth functions whose DSH norm is uniformly bounded and its support is $\Subset U$. Hence, $\chi q_{1/n}\omega_q^{k-p}$ is a sequence of DSH currents whose DSH norm is uniformly bounded and whose support is relatively compact in $U$. From Proposition \ref{prop:equivalencePB} and from our hypothesis, by replacing $\chi$ with another localizing function with smaller support, we have for sufficiently small $\epsilon>0$
\begin{displaymath}
\int_{(z_1)_\epsilon} qS\wedge \omega_q^{k-p}\geq \int \chi qS\wedge \omega_q^{k-p}\geq\lim_{n\to\infty}\int S\wedge \chi q_{1/n}\omega_q^{k-p}> -\infty.
\end{displaymath}
The second inequality is simply due to Fatou's lemma.
\end{proof}

%\begin{remark}
%In the same way, we can prove that if $\cU_S$ is bounded in an open set $W$, then $S$ has no mass on any analytic subset in $W$. 
%\end{remark}

For Definition \ref{def:conti}, we have the following description. Its proof and the proof of its corollary are similar to Proposition \ref{prop:PBforpositiveclosed} and its corollary. So, we omit them.
\begin{proposition}
Let $S\in\cC_p$ and $\cU_S$ a super-potential of $S$ of mean $m$. If $\cU_S$ is continuous in $W$, then for any open subset $\widetilde{W}\Subset W$, $\cU_S$ is continuous on $\cC_{k-p+1}(\widetilde{W})$. The topology on $\cC_{k-p+1}(\widetilde{W})$ is the subspace topology from $\cC_{k-p+1}$.
\end{proposition}

%\begin{proof}
%We simply apply the second part of the proof of Proposition \ref{prop:PCPBcontibdd} exactly in the same way.
%\end{proof}

%\begin{remark} As in Remark \ref{rmk:W}, if there exists an $R_0\in\cC_{k-p+1}(W)$ such that $\cU_S(R_0)$ is finite, then $\cU_S$ is continuous on $\cC_{k-p+1}(W)$.
%\end{remark}

\begin{corollary}
When $p=1$, Definition \ref{def:conti} is equivalent to the local continuity of quasi-potential \emph{functions} of the positive closed $(1, 1)$-current $S$ in the following sense: if a quasi-potential $u$ of $S$ is continuous in an open subset $W$ of $\PP^k$, then the super-potential $\cU_S$ of mean $m$ is continuous in $W$ and conversely, if $\cU_S$ is continuous in $W$, then for any open subset $\widetilde{W}\Subset W$, $u$ is continuous in $\widetilde{W}$.
\end{corollary}

\section{Analytic (Sub-)Multiplicative Cocycles}\label{sec:analcocycle}
The analytic (sub-)multiplicative cocycles were introduced by Favre (\cite{F2000}, \cite{F2000-1}) and further studied by Dinh (\cite{D2009}) and Gignac (\cite{G2014}). See also Ahn (\cite{Ahn2015}).

We summarize Section 2 in \cite{Ahn2015} that will be used to investigate the asymptotic behavior of two different multiplicities related to $f$: one related to the local multiplicity of $f^n$ and the other related to the hypersurface $\Psi_n$ of the critical values of $f^n$. %We summarize definitions and related theorems.

Let $X$ be an irreducible compact complex space of dimension $k$, not necessarily smooth. Let $g:X\to X$ be an open holomorphic map.
\begin{definition}[See Definition 1.1 in \cite{D2009}]
A sequence $\{\kappa_n\}$ of functions $\kappa_n:X\to (0, \infty)$ for $n\geq 0$ is said to be an analytic submultiplicative (resp., multiplicative) cocycle (with respect to $g$) if for all $m, n\geq 0$ and for all $x\in X$,
\begin{enumerate}
\item $\kappa_n$ is usc with respect to the Zariski topology on $X$ and $\kappa_n\geq c_\kappa^n$ for some constant $c_\kappa>0$, and
\item $\kappa_{n+m}(x)\leq \kappa_n(x)\cdot\kappa_m(g^n(x))$ (resp., $=$).
\end{enumerate}
\end{definition} 
\begin{definition}[See Introduction in \cite{D2009}]
$\kappa_{-n}(x)=\max_{y\in g^{-n}(x)}\kappa_n(y)$.
\end{definition}
\begin{theorem}[See Theorem 1.2 in \cite{D2009}]\label{thm:analyticmulticocycle}
The sequence $\{(\kappa_{-n})^{1/n}\}$ converges to a function $\kappa_-$ defined over $X$ with the following properties: for all $\delta>\inf_X\kappa_-$, $\{\kappa_-\geq \delta\}$ is a proper analytic subset of $X$, invariant under $g$ and contained in the orbit of $\{\kappa_n\geq\delta^n\}$ for all $n\geq 0$. In particular, $\kappa_-$ is usc in the Zariski sense.
\end{theorem}

For each $n>0$, define $\mu_n(x)$ to be the local multiplicity of $f^n$ at $x\in\PP^k$. Let $\iota(x, g):=\nu(x, dd^c\log|g|)$ and define $\mu'_n(x):=2k-1+2\iota(x, J_{f^n})$ on $\PP^k$, where $J_{f^n}$ denotes the Jacobian determinant of $f^n$. As discussed in \cite{Ahn2015}, both $\{\mu_n\}$ and $\{\mu'_n\}$ are analytic (sub-)multiplicative cocycles. Let $\mu_-$ and $\mu'_-$ denote their limit functions in Theorem \ref{thm:analyticmulticocycle}. Then, Theorem \ref{thm:analyticmulticocycle} implies
\begin{lemma}[See Lemma 2.6 in \cite{Ahn2015}]\label{lem:analyticmultico} Let $\lambda$ and $\lambda'$ be any real numbers such that $1<\lambda, \lambda'<d$. Define $E:=\{\mu_-\geq d/\lambda\}\cup\{\mu'_-\geq d/\lambda'\}$. Then, there exists an integer $N_E>0$ such that for any integer $j>0$, we have
\begin{enumerate}
\item $E$ is invariant under $f$,
\item $\mu_{-jN_E}(x)<(d/\lambda)^{jN_E}$ for $x\in\PP^k\setminus E$ and
\item $\nu(x, [\Psi_{jN_E}])<c_\Psi jN_E(d^{k+1}/(\lambda^k\lambda'))^{jN_E}$ for $x\in \Psi_{jN_E}\setminus E$
\end{enumerate}
where $\Psi_n$ denotes the hypersurface of the critical values of $f^n$, $[\Psi_n]$ the current of integration on $\Psi_n$ and $c_\Psi$ the number of the irreducible components of $\Psi_1$.
\end{lemma}

\section{Lojasiewicz type inequalities}
In this section, we introduce some materials related to Lojasiewicz type inequalities that will be used in the proof of the main theorem. %These inequalities are useful in dealing with singularities of $f$. 

\begin{proposition}[See Lemma 3.3 in \cite{Ahn2015}]\label{prop:approxQpotential}
Let $V$ be an analytic hypersurface in $\PP^k$ of degree $d_V$ and $[V]$ its current of integration. Let $\varphi_V$ be a quasi-potential of $[V]$, that is, a unique negative q-psh function over $\PP^k$ such that $\sup_{\PP^k}\varphi_V=0$ and $dd^c\varphi_V=[V]-d_V\omega$. Let $0<\delta< d_V$ be given. Assume that there exists a proper analytic subset $E_{V,\delta}$ of $V$ such that for all $P\in V\setminus E_{V,\delta}$, $\nu(P, [V])<\delta$. Then, there are constants $C_{V, \delta}, A_{V, \delta}>0$ such that on $\PP^k$,
\begin{displaymath}
\delta\log\dist(\cdot, V)+C_{V, \delta}\log\dist(\cdot, E_{V, \delta})-A_{V, \delta}\leq \varphi_V\leq\log\dist(\cdot, V)+A_{V, \delta}.
\end{displaymath} 
\end{proposition}

%We use Lojasiewicz type inequalities to construct a special localization function related to an analytic subset and as a result, we can generalize Proposition 2.3.6 in \cite{DS2009}. 
The following lemma can be understood as a generalization of Lemma 2.2.7 in \cite{DS2009} to an arbitrary analytic subset.
\begin{lemma}\label{lem:generalDSHfinitecutoff}
Let $E$ be an analytic subset in $\PP^k$. Then, for a sufficiently small $t_0>0$, there exists a family of smooth functions $\{\chi_{E, t}\}_{0<t<t_0}$ such that $\chi_{E, t}:\PP^k\to [0, 1]$ equal to $1$ on $E_t$ with $\supp \chi_{E, t}\Subset E_{A_Et^{1/D_E}}$ and such that $\|\chi_{E,t}\|_\dsh\leq A_E$, where $A_E, D_E>0$ are constants independent of $t$.
\end{lemma}

We inductively use Lemma 2.2.7 in \cite{DS2009} and Lojasiewicz type inequalities.

\begin{proof}
According to Chow's theorem (for example, see p.167 in \cite{GH}), $E$ can be written as the vanishing locus of a finite number of homogeneous polynomials $\{f_i\}_{i=1}^N$ of $\PP^k$. Let $H_i$ denote the zero set of $f_i$ so that $\cap_{i=1}^N H_i=E$ and $d_i$ the degree of $f_i$ for each $i$. Then, Lemma 2.2.7 in \cite{DS2009} implies that for any sufficiently small $t>0$, there exists a smooth cut-off function $\widehat{\chi_{i, t}}:\PP^k\to[0,1]$ such that $\supp \widehat{\chi_{i, t}}\subset (H_i)_{\widehat{A_i}t^{1/d_i}}$, $\widehat{\chi_{i, t}}\equiv 1$ on $(H_i)_t$ and $\|\chi_{i, t}\|_{\dsh}\leq \widehat{A_i}$ where $\widehat{A_i}$'s are positive constants independent of $t$.

Write $\cH_j:=\cap_{i=1}^j H_i$ for $j=1, \cdots, N$. It suffices to inductively prove the following claim for every $j=1, \cdots, N$:
\medskip

\noindent{\bf Claim:} For any sufficiently small $t>0$, there exists a smooth cut-off function $\chi_{j, t}:\PP^k\to[0, 1]$ such that $\supp \chi_{j, t}\Subset (\cH_j)_{A_jt^{1/D_j}}$, $\chi_{j, t}\equiv 1$ on $(\cH_j)_t$ and $\|\chi_{j, t}\|_\dsh\leq A_j$ where $A_j, D_j>0$ are constants independent of $t$.
\medskip

Lemma 2.2.7 in \cite{DS2009} implies the case of $j=1$. Assume that the claim is true for $j$. Define $\chi_{j+1, t}:=\chi(\chi_{j, t}+\widehat{\chi_{j+1, t}}-1)$ where $\chi$ is the convex increasing function with bounded derivatives as in Lemma 2.2.6 in \cite{DS2009}, and $A':=\max\{A_j, \widehat{A_{j+1}}\}$ and $d':=\max\{D_j, d_{j+1}\}$. Then, $\chi_{j+1, t}\equiv 1$ on $(\cH_{j+1})_t$ and $\chi_{j+1, t}\equiv 0$ on $\PP^k\setminus ((\cH_j)_{A't^{1/d'}} \cap (H_{j+1})_{A't^{1/d'}})$. Also, $\|\chi_{E,t}\|_\dsh\leq cA'$ for a sufficiently large constant $c>0$ is clear.

Now, we apply a Lojasiewicz type inequality to estimate the support of $\chi_{j+1, t}$ in terms of the distance to the set $\cH_{j+1}$. If $z\in (\cH_j)_{A't^{1/d'}} \cap (H_{j+1})_{A't^{1/d'}}$, then, there exists $z'\in H_{j+1}$ such that $\dist(z, z')\leq A't^{1/d'}$. The triangle inequality implies that $\dist(z', \cH_j)\leq 2A't^{1/d'}$. A Lojasiewicz type inequality (for example, see p.14 and p.62 in \cite{Mal}) implies that
\begin{displaymath}
2A't^{1/d'}\geq\dist(z', \cH_j)\geq (A'')^{-1}\dist(z', \cH_{j+1})^M
\end{displaymath}
for some large positive constants $A''$ and $M$ independent of $t$.

Hence, the triangle inequality and the choice of $D_{j+1}=d'M$ and $A_{j+1}=C(A'A'')^{1/M}$ with a sufficiently large constant $C>0$ prove the case of $j+1$.
\end{proof}

In the exactly same way as in the proof of Lemma 2.3.6 in \cite{DS2009} with the family $\{\chi_{E,t}\}$ of cut-off functions in Lemma \ref{lem:generalDSHfinitecutoff}, we obtain the following generalization of Lemma 2.3.6 in \cite{DS2009} and omit the proof.
\begin{proposition}\label{prop:nearE}
Let $E$ be an analytic subset of $\PP^k$ and $T_j$ positive closed $(1, 1)$-currents which admit a $(K, \alpha)$-H\"older continuous quasi-potential $q_j$, that is, a $(K, \alpha)$-H\"older continuous q-psh function on $\PP^k$ such that $T_j=\omega+dd^cq_j$. Let $R\in\cC_p$ be smooth and denote by $U$ its Green quasi-potential. Then, for every sufficiently small $t>0$, we have
\begin{displaymath}
\Bigg|\int_{E_t}U\wedge T_1\wedge\cdots\wedge T_{k-p+1}\Bigg|\leq c_E(K+1)^k t^\beta
\end{displaymath}
where $\beta=(20k^2D_E)^{-k}\alpha^k$, $D_E$ is the constant for $E$ in Lemma \ref{lem:generalDSHfinitecutoff} and $c_E>0$ is a constant independent of $\alpha$, $K$, $t$, and $R$.
\end{proposition}

The definition of $(K, \alpha)$-H\"older continuity follows below:
\begin{definition}
Let $0<\alpha\leq 1$ and $K>0$. A function $g:U\subseteq\PP^k\to\CC$ is said to be $(K, \alpha)$-H\"older continuous on $U$ if for any $x, y\in U$,
\begin{displaymath}
|g(x)-g(y)|\leq K\dist(x,y)^\alpha
\end{displaymath}
holds where $\dist(\cdot, \cdot)$ is the distance with respect to the Fubini-Study metric of $\PP^k$.
\end{definition}

%\begin{remark}
%Notice that in Proposition \ref{prop:nearE}, the smoothness assumption for $R$ is for the definition of the integral and that this estimate is uniform with respect to $R$. So, we can use this proposition in the estimate of super-potentials of general $R\in\cC_p$.
%\end{remark}

%\begin{proof} Replace the cut-off function in the proof of Lemma 2.3.6 in \cite{DS2009} by the one in Lemma \ref{lem:generalDSHfinitecutoff}.
%\end{proof}

\section{Proof of Theorem \ref{thm:mainthm}}\label{sec:proof1}

In this section, we give the proof of our main theorem. We will use similar technical settings to those in \cite{Ahn2015} and follow the notations in \cite{Ahn2015}. 
%Indeed, the essential point is that we adjust the regularization of currents to modify Lemma 9.1 of \cite{Ahn2015} in the same spirit as in Proposition 2.3.6 of \cite{DS2009}. 
Below are some notations, definitions and properties.

In the remaining of the paper, the symbols $\lesssim$ and $\gtrsim$ mean inequalities up to multiplicative constants independent of $n$, $i$, and $\theta$, which will be introduced soon. We denote by
\begin{displaymath}
L:=d^{-p}f^*\quad\textrm{ and }\quad\Lambda:=d^{-p+1}f_*.
\end{displaymath}

Let $S\in \cC_p$ and $T$ the Green current associated with $f$.
\begin{definition}[See Section 5 in \cite{DS2009}] The dynamical super-potential of $S$, denoted by $\cV_S$, is defined by 
\begin{displaymath}
\cV_S:=\cU_S-\cU_{T^p}-c_S,
\end{displaymath}
where $\cU_S$ and $\cU_{T^p}$ are the super-potentials of mean $0$ of $S$ and $T^p$, respectively and $c_S:=\cU_S(T^{k-p+1})-\cU_{T^p}(T^{k-p+1})$. Accordingly, the dynamical Green quasi-potential $V_S$ of $S$ is defined by
\begin{displaymath}
V_S:=U_S-U_{T^p}-(m_S-m_{T^p}+c_S)\omega^{p-1}
\end{displaymath}
where $U_S$ and $U_{T^p}$ are the Green quasi-potentials of $S$ and $T^p$ and $m_S$ and $m_{T^p}$ are their means, respectively.
\end{definition}

%\begin{lemma}[See Lemma 5.4.6 in \cite{DS2009}]\label{lem:546} $\cV_S(T^{k-p+1})=0$, $\cV_S(R)=\langle V_S, R\rangle$ for smooth $R\in\cC_{k-p+1}$ and $\cV_{L(S)}=d^{-1}\cV_S\circ\Lambda$ on $\cC_{k-p+1}$. Moreover $\cU_S-\cV_S$ is bounded by a constant independent of $S$.
%\end{lemma}

%Recall that $\Psi_1$ denotes the set of critical values of $f$.
%\begin{lemma}[See Lemma 5.4.9 in \cite{DS2009}]\label{lem:549} For smooth $R\in\cC_{k-p+1}$, $\cV_S(\Lambda(R))=\langle V_S, \Lambda(R)\rangle_{\PP^k\setminus \Psi_1}$.
%\end{lemma}

As in \cite{Ahn2015}, Proposition 3.1.9 and Proposition 3.2.6 in \cite{DS2009} imply that for the proof of Theorem \ref{thm:mainthm}, it suffices to prove Proposition \ref{prop:mainthm} below.
\begin{proposition}\label{prop:mainthm}
Assume the same $f$, $T$ as in Theorem \ref{thm:mainthm}. There exists a proper (possibly empty) invariant analytic subset $E$ for $f$ if $S\in\cC_p$ admits the super-potential $\cU_S$ of mean $0$ bounded near $E$, then we have $\cV_{L^n(S)}(R)\to 0$ exponentially fast as $n\to\infty$ for any smooth $R\in\cC_{k-p+1}$.
\end{proposition}

\begin{proof}[Proof of Proposition \ref{prop:mainthm}] 
For the proof of the proposition, it suffices to prove the statement for $f^N$ for some $N\in\NN$. Recall the invariant set $E$ constructed in Lemma \ref{lem:analyticmultico}. As in \cite{Ahn2015}, by replacing $f$ by $f^N$ and $d$ by $d^N$ with a sufficiently large $N$, and properly choosing $\delta\in (1, d)$, we may assume that
\begin{enumerate}
\item $\mu_{-1}<\delta$ on $\PP^k\setminus E$ and $\nu(x, [\Psi_1])<\delta$ for $x\in\Psi_1\setminus E$,
\item $E$ is invariant under $f$ and
\item $(20k^2(\delta+1/2))^{8k}<(40k^2\delta)^{8k}<d$.
\end{enumerate}
Here, $\Psi_1$ denotes the hypersurface of the critical values of $f$.
\medskip

Let $S\in\cC_p$ be a current satisfying the hypothesis in Theorem \ref{thm:mainthm} with $E$ as above. We divide $\PP^k$ into three sets for computational purposes.
\begin{lemma}[See Lemma 3.1. in \cite{DS2010}]\label{lem:lowerbound4inverse}
There is a constant $A_1\geq 1$ such that for every $X\subseteq \PP^k$ and for every $Y\subseteq\PP^k$ with $f(Y)\subseteq Y$, we have that for $j\geq 0$,
\begin{displaymath}
\dist(f^{-j}(X), Y)\geq {A_1}^{-j}\dist(X, Y).
\end{displaymath}
\end{lemma}

We recall a corollary with our choice of $E$ and $\delta$.
\begin{corollary}[See Corollary 4.4 in \cite{DS2010}]\label{cor:basicLoja}
There are $N_3\in\NN$ and a constant $c_3\geq 1$ such that for a constant $0<t<1$ and for $x, y\in\PP^k$ with $\dist(x, E), \dist(y, E)$ $>t$, we can write
\begin{displaymath}
f^{-1}(x)=\left\{x_1, \cdots, x_{d^k}\right\}\quad\textrm{ and }\quad f^{-1}(y)=\left\{y_1, \cdots, y_{d^k}\right\}
\end{displaymath}
with $\dist(x_i, y_i)\leq c_3t^{-N_3}\dist(x, y)^{1/\delta}$.
\end{corollary}

Let $\epsilon_S>0$ be such that the superpotential $\cU_S$ of $S$ of mean $0$ is bounded in the neighborhood $E_{\epsilon_S}$ of $E$. Let $\varepsilon_n:=d^{-A_0 n}$ for sufficiently large $A_0>0$ so that $A_E(3d^{-A_0n^2})^{1/D}<\epsilon_S A_1^{-n}$ holds for every $n\geq 1$ where $A_E$ and $D$ are the constants $A_E$ and $D_E$ in Lemma \ref{lem:generalDSHfinitecutoff} for our $E$, and $A_1$ and $\epsilon_S$ as above. For each $n\geq 1$ and $i$ with $0\leq i\leq n$, we define three sequences $\{s_{n, i}\}$, $\{\epsilon_{n, i}\}$ and $\{t_{n, i}\}$ by
\begin{displaymath}
s_{n, i}:=\varepsilon_n^{ni},\quad\epsilon_{n, i}:=\varepsilon_n^{nC(2+N_3)(40k^2\delta)^{6ki}}\quad\textrm{ and }\quad t_{n, i}:=\epsilon_{n, i}^{1/(20k)},
\end{displaymath}
where $N_3$ is the constant in Corollary \ref{cor:basicLoja} and $C>1$ is the constant $C_{V, \delta}$ in Proposition \ref{prop:approxQpotential} with $\Psi_1$, $E\cap V$ and our $\delta$. Notice that for sufficiently large $n$, $s_{n, i}\gg t_{n, i}$. For notational convenience, $V$ denotes the hypersurface $\Psi_1$ of the critical values of $f$ in the rest of the proof.

Now, we define the three sets for each $n$ and $i$ as follows:
\begin{displaymath}
W''_{n, i}:=E_{s_{n, i}}, \quad W_{n, i}:=V_{t_{n, i}}\setminus W''_{n, i}\quad\textrm{ and }\quad W'_{n, i}:=\PP^k\setminus(W_{n, i}\cup W''_{n, i}).
\end{displaymath}
Roughly speaking,
\begin{itemize}
\item on $W_{n, i}$, $f$ and $V$ have low multiplicities but $L^{n-i}(S)$ can be singular,
\item on $W'_{n, i}$, $\Lambda(R')$ is smooth for smooth $R'\in\cC_{k-p+1}$ and
\item on $W''_{n, i}$, $f$ or $V$ can have high multiplicities but the super-potential of $L^{n-1}(S)$ of mean $0$ is bounded.
\end{itemize}

Let $R\in\cC_{k-p+1}$ be smooth and $\{R_{n, i}\}$ a sequence in $\cC_{k-p+1}$ defined by
\begin{displaymath}
R_{n, 0}:=R \textrm{ and } R_{n, i}:=(\Lambda(R_{n, i-1}))_{\epsilon_{n, i}}\textrm{, i.e., the }\epsilon_{n, i}\textrm{-regularization of }\Lambda(R_{n, i-1}).
\end{displaymath}
We denote by $U_j$ and $V_j$ the Green quasi-potential and the dynamical Green quasi-potential of $L^j(S)$, respectively, for $j=0, 1, 2, \cdots$. Then, by Lemma 5.4.6 and Lemma 5.4.9 in \cite{DS2009},
%Lemma \ref{lem:546} and Lemma \ref{lem:549}, 
we have
\begin{align}
\label{eqn:expansion}\cV_{L^n(S)}(R)&=d^{-1}\langle V_{n-1}, \Lambda(R_{n, 0})-R_{n, 1}\rangle_{\PP^k\setminus V}+\cdots\\
%=d^{-n}\cV_S(\Lambda^n(R))=d^{-1}\cV_{L^{n-1}(S)}(\Lambda(R_{n, 0}))\\
%\nonumber&=d^{-1}\langle V_{n-1}, \Lambda(R_{n, 0})-R_{n, 1}\rangle_{\PP^k\setminus V}+d^{-1}\langle V_{n-1}, R_{n, 1}\rangle\\
%\nonumber&=d^{-1}\langle V_{n-1}, \Lambda(R_{n, 0})-R_{n, 1}\rangle_{\PP^k\setminus V}+d^{-1}\cV_{L^{n-1}(S)}(R_{n, 1})\\
%\nonumber&\quad\quad\quad\cdots\\
%\nonumber&=d^{-1}\langle V_{n-1}, \Lambda(R_{n, 0})-R_{n, 1}\rangle_{\PP^k\setminus V}+\cdots\\
\nonumber&\quad+d^{-i}\langle V_{n-i}, \Lambda(R_{n, i-1})-R_{n, i}\rangle_{\PP^k\setminus V}+\cdots\\
\nonumber&\quad+d^{-n}\langle V_0, \Lambda(R_{n, n-1})-R_{n, n}\rangle_{\PP^k\setminus V}+d^{-n}\cV_S(R_{n, n}).
\end{align}
Since $T^p$ has bounded super-potentials and the super-potential of mean $0$ of any current of unit mass is uniformly bounded above by a constant independent of a positive closed current (see Lemma 3.1.2 in \cite{DS2009}), we have
\begin{align*}
d^{-n}\gtrsim \cV_{L^n(S)}(R)
\end{align*}
and therefore, it suffices to estimate lower bounds of $d^{-i}\langle V_{n-i}, \Lambda(R_{n, i-1})-R_{n, i}\rangle_{\PP^k\setminus V}$ for $i=1,\cdots, n$ and $d^{-n}\cV_S(R_{n, n})$ for the proof of the proposition.

In general, $V_{n-i}$ is neither positive nor negative. Theorem \ref{thm:541} implies that there exists a universal $c>0$ such that $V'_{n-i}:=V_{n-i}+U_{T^p}-c\omega^{p-1}=U_{n-i}-c_{n-i}\omega^{p-1}$ is negative in the sense of currents where $\{c_{n-i}\}$ is a bounded sequence. Note that $d_{\cC^{k-p+1}}(\Lambda(R_{n, i-1}), R_{n, i})\lesssim \epsilon_{n, i}$ and $\cU_{T^p}$ is H\"older continuous with respect to $d_{\cC_{k-p+1}}$. Since a super-potential is continuous with respect to the $\theta$-regularization of a current, in order to prove the proposition, it is enough to estimate the following from below and to show that the estimates are uniform with respect to $\theta$ with sufficiently small $|\theta|>0$:
\begin{enumerate}
\item $d^{-i}\langle V'_{\theta, n-i}, \Lambda(R_{n, i-1})-R_{n, i}\rangle_{W_{n, i}\setminus V}$,
\item $d^{-i}\langle V'_{\theta, n-i}, \Lambda(R_{n, i-1})-R_{n, i}\rangle_{W'_{n, i}}$,
\item $d^{-i}\langle V'_{\theta, n-i}, \Lambda(R_{n, i-1})-R_{n, i}\rangle_{W''_{n, i}\setminus V}$ and
\item $d^{-n}\cV_S(R_{n, n})$.
\end{enumerate}
where $V'_{\theta, n-i}=U_{(L^{n-i}(S))_\theta}-c_{n-i}\omega^{p-1}$, $U_{(L^{n-i}(S))_\theta}$ is the Green quasi-potential of $(L^{n-i}(S))_\theta$ and the subscript $\theta$ means the $\theta$-regularization of a current. The mean of $U_{(L^{n-i}(S))_\theta}$ converges to the mean of $U_{n-i}$ as $\theta$ converges to $0$. For the proof of the proposition, we directly use Lemma 7.13, Lemma 8.2, and Lemma 10.2 in \cite{Ahn2015} to (1), (2) and (4), respectively. For all sufficiently large $n$, we have $(1)\gtrsim -\varepsilon_n^{ni}$, $(2)\gtrsim -\varepsilon_n^{ni}$, and $(4)\gtrsim nd^{-n/4}\log\varepsilon_n$. So, we only estimate $(3)$.
\medskip

Let $C_S, C_{T^p}>0$ be two constants such that $\cU_S$ is bounded near $E$ with the constant $C_S$ as in Definition \ref{def:bounded} and $|\cU_{T^p}|<C_{T^p}$ over $\cC_{k-p+1}$.

\begin{lemma}\label{lem:nearE}For all sufficiently large $n$,
\begin{align*}
(3)= d^{-i}\langle V'_{\theta, n-i}, \Lambda(R_{n, i-1})-R_{n, i}\rangle_{W''_{n, i}\setminus V}\gtrsim-d^{-n/8}(C_S+C_{T^p})
\end{align*}
is true for all $\theta$ with sufficiently small $|\theta|$.
\end{lemma}
The proof consists of a series of lemmas. We can write
\begin{align*}
d^{-i}\langle V'_{\theta, n-i}, \Lambda(R_{n, i-1})-R_{n, i}\rangle_{W''_{n, i}\setminus V}&=d^{-i}\langle U_{(L^{n-i}(S))_\theta}, \Lambda(R_{n, i-1})-R_{n, i}\rangle_{W''_{n, i}\setminus V}\\
&\quad\quad-d^{-i}\langle c_{n-i}\omega^{p-1}, \Lambda(R_{n, i-1})-R_{n, i}\rangle_{W''_{n, i}\setminus V}.
\end{align*}

We first consider $d^{-i}\langle U_{(L^{n-i}(S))_\theta}, \Lambda(R_{n, i-1})-R_{n, i}\rangle_{W''_{n, i}\setminus V}$. Let $c_\aut>0$ be a constant such that for all sufficiently small $\epsilon_o>0$ and for all $\zeta$ with $\|\zeta\|_A<1$, we have $\dist(\tau_{\epsilon_o\zeta}(x), x)<c_\aut\epsilon_o$ for every $x\in\PP^k$. Define a neighborhood of $E$ by 
\begin{displaymath}
W''_{n,i,1}:= E_{s_{n, i}-c_\aut\epsilon_{n, i}}.
\end{displaymath}
For notational simplicity, we write $U_{n, i, \theta}$ for $U_{(L^{n-i}(S))_\theta}$, and $R$ for $\Lambda(R_{n, i-1})$. We will recall this notation if necessary. Then, since $U_{n, i, \theta}$ is smooth, we have
\begin{align}
\nonumber&d^{-i}\langle U_{n, i, \theta}, R-R_{\epsilon_{n, i}}\rangle_{W''_{n,i}\setminus V}=d^{-i}\langle U_{n, i, \theta}, R-R_{\epsilon_{n, i}}\rangle_{W''_{n,i}}\\
%\nonumber&=d^{-i}\int_{W''_{n,i}}U_{n, i, \theta}\wedge R-d^{-i}\int_{W''_{n,i}}\int_{\|\zeta\|_A<1}U_{n, i, \theta}\wedge (\tau_{\epsilon_{n, i}\zeta})^* (R) \,d\rho(\zeta)\\
%\nonumber&=d^{-i}\int_{W''_{n,i}}U_{n, i, \theta}\wedge R-d^{-i}\int_{\|\zeta\|_A<1}\int_{W''_{n,i}}U_{n, i, \theta}\wedge (\tau_{\epsilon_{n, i}\zeta})^* (R) \,d\rho(\zeta)\\
%\nonumber&=d^{-i}\int_{W''_{n,i}}U_{n, i, \theta}\wedge R-d^{-i}\int_{\|\zeta\|_A<1}\int_{\tau_{\epsilon_{n, i}\zeta}(W''_{n,i})}(\tau_{\epsilon_{n, i}\zeta})_*(U_{n, i, \theta})\wedge R \,d\rho(\zeta)\\
\label{eqn:region1}&=d^{-i}\int_{\|\zeta\|_A<1}\left(\int_{W''_{n,i,1}}(U_{n, i, \theta}-(\tau_{\epsilon_{n, i}\zeta})_*(U_{n, i, \theta}))\wedge R \right)d\rho(\zeta)\\
\label{eqn:region2}&\quad\quad\quad +d^{-i}\int_{W''_{n, i}\setminus W''_{n,i,1}} U_{n, i, \theta}\wedge R\\
\label{eqn:region3}&\quad\quad\quad-d^{-i}\int_{\|\zeta\|_A<1}\int_{\tau_{\epsilon_{n, i}\zeta}(W''_{n,i})\setminus W''_{n,i,1}}(\tau_{\epsilon_{n, i}\zeta})_*(U_{n, i, \theta})\wedge R \,d\rho(\zeta).
\end{align}

We estimate \eqref{eqn:region1}. Let $\alpha$ denote the H\"older exponent of a quasi-potential of the Green current $T$ associated with $f$.
\begin{lemma}\label{lem:region1}For sufficiently large $n$,
\begin{displaymath}
\eqref{eqn:region1}\gtrsim-d^{-n/8}(C+2C_{T^p}).
\end{displaymath}
is true for all $\theta$ with sufficiently small $|\theta|$.
\end{lemma}

\begin{proof} Let $W:=E_{2s_{n, i}}$. We take $\theta$ so that $c_\aut |\theta|<d^{-A_0n^2}$.
\begin{align}
&\nonumber d^{-i}\int_{W''_{n,i,1}}(U_{n, i, \theta}-(\tau_{\epsilon_{n, i}\zeta})_*(U_{n, i, \theta}))\wedge R\\
&\label{eq:72.1}=d^{-i}\int_{z\in W''_{n,i,1}} \int_{w\in W\setminus\{z\}}(L^{n-i}(S))_\theta(w)\wedge[K(w, z)-K(w, (\tau_{\epsilon_{n, i}\zeta})^{-1}(z))]\wedge R(z)\\
&\label{eq:72.2}\quad\quad+d^{-i}\int_{z\in W''_{n,i,1}} \int_{\PP^k\setminus W}(L^{n-i}(S))_\theta(w)\wedge[K(w, z)-K(w, (\tau_{\epsilon_{n, i}\zeta})^{-1}(z))]\wedge R(z).
\end{align}

%$\theta$ should satisfy $\supp dd^c\Lambda^{n-i}(\varphi_\theta)\Subset W_E$.
We first compute the integral \eqref{eq:72.1}. By the negativity, \eqref{eq:72.1} is bounded by 
\begin{align*}
d^{-i}\int_{z\in W''_{n,i,1}} \left(\int_{w\in W\setminus\{z\}}(L^{n-i}(S))_\theta(w)\wedge K(w, z)\right)\wedge R(z).
\end{align*}
We estimate this integral. Take $M_n=d^{7n/8}$. As in Proposition 2.3.6 in \cite{DS2009}, we define $\eta_n:=\min\{0, M_n+\eta\}$ and $K_n:=-{M_n}\Theta$ and $K_n':=\eta_n\Theta$ where $\eta$ is the function in Proposition \ref{prop:kernel}. Then, $\|\eta_n\|_\dsh$ is bounded independently of $n$, $K_n+K_n'\leq K$ in the sense of currents. Define
\begin{displaymath}
U_n(z):=\int_{w\neq z} K_n(w, z)\wedge R(w)\quad\textrm{ and }\quad U'_n(z):=\int_{w\neq z} K_n'(w, z)\wedge R(w).
\end{displaymath}
Note that $U_n$ is negative closed of mass $\simeq M_n$ and that $U_n+U'_n\leq U_R$ in the sense of currents where $U_R$ is the Green quasi-potential of $R$.

We take $\chi_{n, i}(z):=\chi_{E, 2s_{n, i}}(z)$ for each $n$ and $i$ where the function $\chi_{E, 2s_{n, i}}(z)$ is from Lemma \ref{lem:generalDSHfinitecutoff}. Note that $\|\chi_{n, i}\|_\dsh$ is bounded independently of $n$ and $i$, and that $\supp \chi_{n, i}\Subset E_{A_E{(2s_{n, i})^{1/D}}}$. Then, $\chi_{n, i} U_n$ is a negative DSH current and $\|\chi_{n, i}U_n\|_\dsh\lesssim M_n$. For simplicity, denote by $\varphi=\chi_{n, i} U_n$. Recall $\|dd^c\varphi\|_\var:=\inf\{\|R_+\|: dd^c\varphi=R_+-R_-, R_\pm$ are positive closed.$\}$. We have
\begin{align*}
|\langle (L^{n-i}(S))_\theta-(T^p)_\theta,  \varphi\rangle|&=d^{-(n-i)}|\langle S-T^p, \Lambda^{n-i}(\varphi_\theta) \rangle|\\
&=d^{-(n-i)}|\langle U_S-U_{T^p}, dd^c\Lambda^{n-i}(\varphi_\theta) \rangle|
%&\leq d^{-(n-i)}C\|dd^c\Lambda^{n-i}(\varphi_\theta)\|_\var\\
%&=d^{-(n-i)}C\|dd^c(\varphi_\theta)\|_\var\approx d^{-(n-i)} CM_n
\end{align*}
where $U_S$ and $U_{T^p}$ are the Green quasi-potentials of $S$ and $T^p$, respectively.

Let $R_\pm$ be positive closed currents of the same mass such that $dd^c\varphi=R_+-R_-$ and $\|dd^c\varphi\|_\var=\|R_+\|$. Then, due to the commutativity between $dd^c$ and pull-backs and push-forwards, $dd^c\Lambda^{n-i}(\varphi_\theta)=\Lambda^{n-i}((R_+)_\theta)-\Lambda^{n-i}((R_-)_\theta)$, where $\Lambda^{n-i}((R_\pm)_\theta)$ are positive closed and of the same mass as $\|R_\pm\|$. Due to Lemma \ref{lem:lowerbound4inverse} and our choice of $\chi_{n, i}(z)$ and $\theta$, $\supp dd^c\Lambda^{n-i}(\varphi_\theta)\subset E_{\epsilon_S}$ and $\cU_S$ is bounded over $E_{\epsilon_S}$. We have
\begin{align}
\nonumber&d^{-(n-i)}|\langle U_S-U_{T^p}, dd^c\Lambda^{n-i}(\varphi_\theta) \rangle|\\
\nonumber&\quad\quad=d^{-(n-i)}|\langle U_S-U_{T^p}, \Lambda^{n-i}((R_+)_\theta)-\Lambda^{n-i}((R_-)_\theta) \rangle|\\
\nonumber&\quad\quad=d^{-(n-i)}|\langle U_S, \Lambda^{n-i}((R_+)_\theta)-\Lambda^{n-i}((R_-)_\theta) \rangle|\\
\nonumber&\quad\quad\quad\quad\quad+d^{-(n-i)}|\langle U_{T^p}, \Lambda^{n-i}((R_+)_\theta) \rangle|+d^{-(n-i)}|\langle U_{T^p}, \Lambda^{n-i}((R_-)_\theta) \rangle|\\
\nonumber&\quad\quad\leq d^{-(n-i)}(C_S+2C_{T^p})\|\Lambda^{n-i}((R_+)_\theta)\|=d^{-(n-i)}(C_S+2C_{T^p})\|R_+\|\\
\label{ineq:I}&\quad\quad\leq d^{-(n-i)}(C_S+2C_{T^p})\|\varphi\|_\dsh\lesssim d^{-(n-i)}(C_S+2C_{T^p})M_n.
\end{align}
When we are applying our definition of boundedness in the above, we use the standard regularization of currents.
%Notice that our definition of boundedness is given only for smooth currents, but Since $(R_\pm)_\theta$ are smooth, $\cU_S$ and $\cU_{T_p}$ are finite at $\Lambda^{n-i}((R_\pm)_\theta)$ and therefore, we can use the regularization. (cf. Section \ref{sec:super-potentials})
\smallskip

We use the H\"older continuity of the quasi-potential of $T$ to estimate
\begin{align*}
|\langle (T^p)_\theta, \chi_{n, i} U_n\rangle| = \left|\int (T^p)_\theta\wedge\left( \chi_{n, i} \int K_n\wedge R\right)\right|.
\end{align*}

Here, $(T^p)_\theta$ is smooth. It is not difficult to see that the sequence $\chi_{n, i} (\int K_n\wedge R_{\theta'})$ converges $\chi_{n, i} (\int K_n\wedge R)$ as $\theta'\to 0$ in $\dsh^{k-p}$ where $R_{\theta'}$ is the $\theta'$-regularization of $R$. Hence, by regularizing $R$, we may assume that $R$ is smooth. We have
\begin{align}
\nonumber\left|\int (\tau_{\theta\zeta})^*(T^p)\wedge \left(\chi_{n, i} \int K_n\wedge R\right)\right|&\leq\left|\int_{E_{A_E(2s_{n, i})^{1/D}}}U_n\wedge (\tau_{\theta\zeta})^*(T^p)\right|\\
\nonumber\lesssim M_n({s_{n, i}}^{1/D})^{(\alpha/(D10k^2))^k}&=M_n\left(\varepsilon_n^{\alpha^k/(D^{k+1}(10k^2)^k)}\right)^{ni}\\
\nonumber&=d^{7n/8}(d^{-A_0 n\alpha^k/(D^{k+1}(10k^2)^k)})^{ni}\\
\label{ineq:II}&\leq d^{(7/8)n-[A_0\alpha^k/(D^{k+1}(10k^2)^k)] n^2}.
\end{align}
This estimate is obtained exactly in the same way as in Lemma 2.3.7 and Lemma 2.3.8 in \cite{DS2009} with our cut-off functions in Lemma \ref{lem:generalDSHfinitecutoff} and with $M$ replaced by $M_n$. So, we omit these computations. Notice that this estimate is uniform with respect to $\theta$ if $|\theta|$ is sufficiently small.
\smallskip

From \eqref{ineq:I} and \eqref{ineq:II} in the above, for sufficiently large $n$ and for $\theta$ with sufficiently small $|\theta|$, 
\begin{align}
\label{eq:U_n}d^{-i}|(\langle L^{n-i}(S))_\theta, \chi_{n, i} U_n\rangle|\lesssim d^{-n}M_n(C+2C_{T^p})\lesssim d^{-n/8}(C_S+2C_{T^p}).
\end{align}
\medskip

Now, we estimate $\langle \int (L^{n-i}(S))_\theta\wedge K'_n, R\rangle$. 
%Assume that $n$ is sufficiently large so that $\eta-\log\dist(\cdot, D)$ is uniformly bounded by a constant $c_\eta>0$ as in Proposition \ref{prop:kernel}. Then, we see that $\supp \eta_n$ is relatively compact in the $c_\eta'$-neighborhood of the diagonal of $D$ of $\PP^k\times\PP^k$ for some fixed constant $c_\eta'>0$. Further 
Assume sufficiently large $n$ so that over the support of $\eta_n$ (or equivalently over the support of $K_n'$), we can use the estimate of the singularities of $K$ in Proposition \ref{prop:kernel}. The point of this estimate is that $d$ is very large compared to $\delta$. Recall that $R=\Lambda(R_{n, i-1})$. Then, $\Lambda(R_{n,i-1})\lesssim \|R_{n,i-1}\|_\infty(f_*\omega)^{k-p+1}$ in the sense of currents. Hence, due to the H\"older continuity of the quasi-potential of $f_*(\omega)$, we have 
\begin{displaymath}
\left|\int \left( \int (L^{n-i}(S))_\theta\wedge K'_n\right)\wedge (f_*\omega)^{k-p+1}\right|\lesssim\exp({-(10k^2)^kd^{-k^2}M_n/2}).
\end{displaymath}
This estimate can be carried out in the exactly same way as in Lemma 2.3.9 and Lemma 2.3.10 in \cite{DS2009}. So, we skip the details.

Since $(40k^2\delta)^{7kn}<d^{7n/8}$ from our choice of $\delta$ and $\|R_{n,i-1}\|_\infty\lesssim \epsilon_{n, i-1}^{-5k^2}=\exp(5k^2n^2C(2+N_3)(40k^2\delta)^{6ki}\log d)$, we have
\begin{align}
&\nonumber\left|\int \left(\int(L^{n-i}(S))_\theta\wedge K'_n\right)\wedge R\right|\\
&\quad\nonumber\lesssim \|R_{n, i-1}\|_\infty\left|\int \left(\int(L^{n-i}(S))_\theta\wedge K'_n\right)\wedge (f_*\omega)^{k-p+1}\right|\\
&\quad\nonumber\lesssim \exp(-[(10k^2)^kd^{-k^2}(40k^2\delta)^{kn}/2-5(\log d) k^2n^2C(2+N_3)](40k^2\delta)^{6kn})\\
&\quad<\exp(-n(40k^2\delta)^{6kn})\label{eq:U_n'}
\end{align}
for $\theta$ with sufficiently small $|\theta|$ and for sufficiently large $n$.

We use the negativity of the integrand and the arguments in Lemma \ref{lem:doublecurrent} to estimate \eqref{eq:72.1} from \eqref{eq:U_n} and \eqref{eq:U_n'}:
\begin{align*}
\eqref{eq:72.1}&\leq d^{-i}\int_{z\in W''_{n,i,1}} \left(\int_{w\in W\setminus\{z\}}(L^{n-i}(S))_\theta(w)\wedge K(w, z)\right)\wedge R(z)\\
&\leq d^{-i}\int_{w\in W}(L^{n-i}(S))_\theta(w)\wedge\left(\int_{z\neq w} K(w, z)\wedge R(z)\right)\lesssim d^{-n/8}(C_S+2C_{T^p})
\end{align*}
for sufficiently large $n$ and for $\theta$ with sufficiently small $|\theta|$.

We estimate \eqref{eq:72.2}. Since $\epsilon_{n, i}\ll s_{n, i}$ for sufficiently large $n$, if $\dist(w, z)>s_{n, i}$, then $\dist(w, (\tau_{\epsilon_{n, i}\zeta})^{-1}(z))>s_{n, i}/2$. Let $D^c_{s_{n, i}/2}:=(\PP^k\times\PP^k)\setminus D_{s_{n, i}/2}$. On $D^c_{s_{n, i}/2}$, $K(\cdot, \cdot)$ is smooth and Proposition \ref{prop:kernel} implies that for $(w, z)$ with $\dist(w, z)>s_{n, i}$,
\begin{displaymath}
\|K(w, z)-K(w, (\tau_{\epsilon_{n, i}\zeta})^{-1}(z))\|_{\infty, D^c_{s_{n, i}}}\lesssim \|\nabla K\|_{\infty, D^c_{s_{n, i}/2}}\epsilon_{n, i}\approx s_{n, i}^{1-2k}\epsilon_{n, i}.
\end{displaymath} 
So, we obtain that \eqref{eq:72.2} $\gtrsim -s_{n, i}^{1-2k}\epsilon_{n, i}$. Note that $s_{n,i}\gg\epsilon_{n,i}$ for sufficiently large $n$. Hence, from the estimates of \eqref{eq:72.1} and \eqref{eq:72.2}, our lemma is proved.
\end{proof}

Since the computations in \eqref{eqn:region2} applies to \eqref{eqn:region3} exactly in the same way, we only estimate \eqref{eqn:region2}.
\begin{lemma}\label{lem:73}For all sufficiently large $n$,
\begin{displaymath}
\eqref{eqn:region2}\gtrsim -\varepsilon_n^{ni}.
\end{displaymath}
\end{lemma}
\begin{proof}
We consider
\begin{align}
&\nonumber\int_{W''_{n, i}\setminus W''_{n,i,1}} U_{n, i, \theta}\wedge \Lambda(R_{n, i-1})\\
&\label{eq:79.1}\quad\quad\quad=\int_{(W''_{n, i}\setminus W''_{n,i,1})\cap W_{n,i}} U_{n, i, \theta}\wedge \Lambda(R_{n, i-1})\\
&\label{eq:79.2}\quad\quad\quad\quad+\int_{(W''_{n, i}\setminus W''_{n,i,1})\cap W'_{n,i}} U_{n, i, \theta}\wedge \Lambda(R_{n, i-1}).
\end{align}
Proposition 7.6 in \cite{Ahn2015} implies that $\eqref{eq:79.1}\gtrsim -\varepsilon_n^{ni}$. So, we consider \eqref{eq:79.2}. Observe that $\Lambda(R_{n, i-1})$ is smooth in $W'_{n, i}$ and Lemma 8.1 in \cite{Ahn2015} and Proposition \ref{prop:regCurrents} imply that its sup-norm in $W'_{n, i}$ is bounded by $c'\epsilon_{n, i-1}^{-2k^2-4k}t_{n, i}^{-4k}$ where $c'>0$ is a fixed constant independent of $\theta$, $n$, and $i$. Next, we estimate the following integral. By the negativity of the integrand and by Lemma \ref{lem:doublecurrent}, we have
\begin{align*}
&\int_{(W''_{n, i}\setminus W''_{n,i,1})\cap W'_{n,i}} U_{n, i, \theta}\wedge \omega^{k-p+1}\geq \int_{(W''_{n, i}\setminus W''_{n,i,1})} U_{n, i, \theta}\wedge \omega^{k-p+1}\\
%&\quad\quad\quad=\int_{z\in(W''_{n, i}\setminus W''_{n,i,1})} \int_{w\neq z}(L^{n-i}(S))_\theta(w)\wedge K(w, z)\wedge \omega^{k-p+1}(z)\\
&\quad\quad\quad=\int_{w\in \PP^k} (L^{n-i}(S))_\theta(w)\wedge\Bigg(\int_{z\in W''_{n, i}\setminus (W''_{n,i,1}\cup\{w\})} K(w, z)\wedge \omega^{k-p+1}(z)\Bigg).
\end{align*}
It is not difficult to see that $\int_{z\in W''_{n, i}\setminus (W''_{n,i,1}\cup\{w\})} K(w, z)\wedge \omega^{k-p+1}(z)$ is a bounded form of $w$ from Proposition \ref{prop:kernel}. We want to use Proposition \ref{prop:kernel} in order to estimate an upper bound of the below for a given $w$:
\begin{align}
\nonumber&\left\|\int_{z\in(W''_{n, i}\setminus (W''_{n,i,1}\cup\{w\})} K(w, z)\wedge \omega^{k-p+1}(z)\right\|_{pt}\\
%&\quad\leq \int_{z\in W''_{n, i}\setminus (W''_{n,i,1}\cup\{w\})} \|K(w, z)\| \omega^k(z)\\
%\nonumber&\lesssim \int_{z\in W''_{n, i}\setminus (W''_{n,i,1}\cup\{w\})} -(\log \dist(w, z))\dist(w, z)^{2-2k}\omega^k(z)\\
&\label{eq:79.1.1}\lesssim \int_{z\in [W''_{n, i}\setminus (W''_{n,i,1}\cup\{w\})]\cap B_w(\epsilon_{n, i}^{1/N})} -(\log \dist(w, z))\dist(w, z)^{2-2k}\omega^k(z)\\
&\label{eq:79.1.2}\quad\quad\quad +\int_{z\in (W''_{n, i}\setminus W''_{n,i,1})\setminus B_w(\epsilon_{n, i}^{1/N})} -(\log \dist(w, z))\dist(w, z)^{2-2k}\omega^k(z)
\end{align}
where $N=10k$ is a fixed constant.

The integral \eqref{eq:79.1.2} is bounded as follows:
\begin{align*}
&\int_{z\in (W''_{n, i}\setminus W''_{n,i,1})\setminus B_w(\epsilon_{n, i}^{1/N})} -(\log \dist(w, z))\dist(w, z)^{2-2k}\omega^k(z)\\
&\quad\quad\quad\leq -(\log \epsilon_{n, i}^{1/N})\epsilon_{n, i}^{(2-2k)/N}\cdot \textrm{ the volume of }(W''_{n, i}\setminus W''_{n,i,1})\\
&\quad\quad\quad\lesssim -(\log \epsilon_{n, i}^{1/N})\epsilon_{n, i}^{1-(2k-2)/N}\lesssim \epsilon_{n, i}^{1/2}.
\end{align*}
For the estimate of the volume of $W''_{n, i}\setminus W''_{n,i,1}$, we use the fact that the singular set of an analytic subset of $\PP^k$ is an analytic subset of $\PP^k$ as well and Weyl's tube formula with multiplicities counted in Section 3(c) of \cite{Griffiths}. %(For the finiteness, see Section 3(c).)
 
Assume hypothetically that \eqref{eq:79.1.1} $\lesssim \epsilon_{n,i}^{1/2}$. Then, together with the discussion in the above about \eqref{eq:79.1}, \eqref{eq:79.2}, \eqref{eq:79.1.1} and \eqref{eq:79.1.2}, we have
\begin{align*}
\int_{(W''_{n, i}\setminus W''_{n,i,1})\cap W'_{n,i}} U_{n, i, \theta}\wedge \Lambda(R_{n, i-1})\gtrsim-\epsilon_{n, i-1}^{-2k^2-4k}t_{n, i}^{-4k}\epsilon_{n, i}^{1/2}\gtrsim-\varepsilon_n^{ni}
\end{align*}
for all $\theta$ with sufficiently small $|\theta|$ and for all sufficiently large $n$. This completes the proof of Lemma \ref{lem:73}.

So, it remains to estimate \eqref{eq:79.1.1} and it will be completed in the next lemma.
\end{proof}

The integral \eqref{eq:79.1.1} can be computed as follows.
\begin{lemma}\label{lem:geotech} For all sufficiently large $n$,
\begin{displaymath}
\int_{z\in [W''_{n, i}\setminus (W''_{n,i,1}\cup\{w\})]\cap B_w(\epsilon_{n, i}^{1/N})} -(\log \dist(w, z))\dist(w, z)^{2-2k}\omega^k(z)\lesssim \epsilon_{n,i}^{1/2}.
\end{displaymath}
\end{lemma}
The basic idea is to follow the strategy in Lemma 2.3.7 in \cite{DS2009}. The difficulty in our case is that the set $\partial E_{t}$ for $s_{n, i}-c_\aut\epsilon_{n, i}<t<s_{n, i}$ may not have a good structure. Instead, we will use the conditions that $\epsilon_{n, i}\ll s_{n, i}$ and that $\epsilon_{n, i}$ shrinks exponentially faster than $s_{n, i}$ as $n\to\infty$. 

\begin{proof}Without loss of generality, we may assume that $w\in W''_{n, i}\setminus W''_{n, i, 1}$. Indeed, other cases can be easily shown to be bounded by this case. Let $t:=\dist(w, E)$. The set $E':=\overline{E_{t-\epsilon_{n, i}^2}}$ is compact and the collection $\cup_{\zeta\in E}B_\zeta(t)$ of open balls forms a covering of $E'$. Hence, we can find a finite subcover $G_w:=\cup_{l=1}^{n_w}B_{\zeta_l}(t)$ of $E'$ from $\cup_{\zeta\in E}B_\zeta(t)$. Observe that $\dist(w, \partial G_w)<\epsilon_{n, i}^2$, $\partial G_w$ consists of a finite union of smooth real hypersurfaces and its singular points are nowhere dense in $\partial G_w$.

Let $\zeta\in \partial G_w$ be such that $\dist(\zeta, w)<\epsilon_{n, i}^2$. Let $r<2\epsilon_{n, i}^{1/N}$ be a constant. Notice that $r\ll s_{n, i}$. We estimate the surface area of the smooth part of the piece $P_r:=B_\zeta(r)\cap \partial G_w$ of the real surface $\partial G_w$. Notice that at every point $b\in\partial G_w$, we can find a ball $B$ of radius $s_{n, i}/3$ which is $\gg r$ such that $b\in \partial B$ and $B\subset G_w$. In particular, if $b$ is a smooth point of $\partial G_w$, $B$ should be tangent to $\partial G_w$ at $b$.

% working here

For each smooth point $p$ of $P_r$, we attach an inward normal half ray which is uniquely determined due to the real codimension being $1$ and define $h(p)$ to be the intersection point of this half ray and the sphere $\partial B_\zeta(r)$. Two different half rays can meet only outside a ball of radius $s_{n, i}/3$. Indeed, if there exists another smooth point $q\in P_r\setminus\{p\}$ such that $h(q)=h(p)$. Then, either the ball of radius $s_{n, i}/3$ tangent at $p$ or $q$ must contain the other point by the triangle inequality, which is a contradiction. Hence, the map $h$ from the regular points of $P_r$ to $\partial B_\zeta(r)$ is injective. It is not difficult to see that $h$ is smooth on the regular points of $P_r$. Since $r\ll s_{n, i}$ and $\epsilon_{n, i}$ shrinks exponentially faster than $s_{n, i}$ as $n\to \infty$, it is not difficult to see that the surface area of the regular points of $P_r$ is bounded by $2\cdot$the surface area of the sphere  $\partial B_\zeta(r)$, which is a fixed constant multiple of $r^{2k-1}$. (Here, for the surface area, for example, we take a finite atlas of $\PP^k$ and use the Euclidean metric in each coordinate chart.) 

We cover $P:=P_{2\epsilon_{n, i}^{1/N}}$ for our $\zeta$ with balls of radius $2\epsilon_{n, i}$ as follows. Let $A$ be a maximal subset of $P$ such that the distance between two points in $A$ is $\geq \epsilon_{n, i}$. Then, the balls of radius $2\epsilon_{n, i}$ with centers at points in the set $A$ cover $P$ and the ones of radius $3c_A\epsilon_{n, i}$ cover $[W''_{n, i}\setminus W''_{n,i,1}]\cap B_\zeta(2\epsilon_{n, i}^{1/N})$ where $c_A:=\max\{1, c_\aut\}$. 

For the remaining, we just apply the same computations as in Lemma 2.3.7 in \cite{DS2009} and then we obtain the desired estimate.

\end{proof}

\begin{proof}[Proof of Lemma \ref{lem:nearE}]
For the part $d^{-i}\langle c_{n-i}\omega^p, \Lambda(R_{n, i-1})-R_{n, i}\rangle_{W''_{n, i}\setminus V}$, we use similar arguments to Lemma \ref{lem:region1} and Lemma \ref{lem:73} but they become much simpler. So, we omit the details. Together with the discussion in the above, Lemma \ref{lem:region1} and Lemma \ref{lem:73} verify the lemma.
\end{proof}

Now, we complete the proof of Proposition \ref{prop:mainthm}. Recall the discussion in the beginning of the proof. Lemma 7.13, Lemma 8.2, Lemma 10.2 in \cite{Ahn2015} and Lemma \ref{lem:nearE} implies
\begin{align*}
d^{-n}\gtrsim \eqref{eqn:expansion}&\gtrsim -nd^{-n/8}(C+C_{T^p})+nd^{-n/4}\log \varepsilon_n\\
&\quad\gtrsim -nd^{-n/8}(C+C_{T^p})-n^2d^{-n/4}(A_0\log d)
\end{align*}
for all sufficiently large $n$. This completes the proof.
\end{proof}

%\section{Questions}
%\noindent CHECK: Due to the positivity, I believe that I can replace $s_I$ by a large positive constant.
%Let $\varphi$ is a test form whose support is in a neighborhood of $I_-$.
%
%\noindent Question: Show the equidistribution for bounded super-potential near $E$. Here, one point is that we may not be able to check the speed of convergence. Think about this. It seems possible. As $\Lambda(\omega)$ is unbounded along the critical values, this is not so trivial. This seems to be the good next goal.

\section{Regular polynomial automorphisms of $\CC^k$}\label{sec:poly}
%In this section, we consider equidistribution of positive closed currents for regular polynomial automorphisms of $\CC^k$.
Let $f$ be a polynomial automorphism of $\CC^k$ of algebraic degree $d_+\geq 2$ with its inverse $f^{-1}$ of algebraic degree $d_-$. We replace $f$ and $f^{-1}$ by their natural extensions to birational maps on $\PP^k$. Let $I_\pm$ denote the indeterminancy sets for $f^{\pm 1}$, respectively, and $L_\infty$ the hyperplane at infinity. We say that $f$ is regular (in the sense of \cite{sibony}) if $I_+\cap I_-=\emptyset$. The indeterminancy sets $I_\pm$ are irreducible, $I_\pm\subset L_\infty$ and there is an integer $p>0$ such that
\begin{displaymath}
\dim I_+=k-p-1\quad\textrm{ and }\quad \dim I_-=p-1.
\end{displaymath}
We define
\begin{displaymath}
K_\pm:=\{z\in\CC^k: \{f^{\pm n}(z)\}_{n\geq 0} \textrm{ is bounded}\}.
\end{displaymath}
Below are some well-known properties.
\begin{proposition}\label{prop:polyautobasic} Assume $f$ and $f^{-1}$ are as above.
\begin{enumerate}
\item $f(L_\infty\setminus I_+)=f(I_-)=I_-$ and $f^{-1}(L_\infty\setminus I_-)=f^{-1}(I_+)=I_+$;
\item $d_+^p=d_-^{k-p}$;
\item $\overline{K_\pm}=K_\pm\cup I_\pm$ in $\PP^k$;
\item $I_\mp$ is attracting for $f^\pm$ and $\PP^k\setminus \overline{K_\pm}$ is its attracting basin;
\item The positive closed $(1, 1)$-currents $d_\pm^{-n}(f^{\pm n})^*(\omega)$ converge to the Green $(1, 1)$-currents $T_\pm$ associated with $f_\pm$. The quasi-potentials of $T_\pm$ are H\"older continuous outside $I_\pm$;
\item $f^*(T_+)=d_+T_+$ and $f_*(T_-)=d_-T_-$.
\end{enumerate}
\end{proposition}

Let $1\leq s\leq p$. For notational convenience, we denote $L_s:=d_+^{-s}f^*$ and $\Lambda_s:=d_+^{-s+1}f_*$. Let $S\in\cC_s$ be such that its super-potential $\cU_S$ of mean $0$ is bounded near $I_-$. Formally, write $S_0:=S$ and $S_n:=L_s^n(S)$ and they are actually well-defined as below. For the definitions in this case, see \cite{DS2009}. The proofs are all straightforward.%Note that Proposition \ref{prop:nomass} implies that the current $S$ has no mass on $I_-$.

%\begin{proof} Let $u$ be a quasi-psh function such that $u=-\infty$ on $E$ and $u$ is smooth outside $E$. Fix a smooth cut-off function $\chi$. Note that $dd^c(\chi u\omega^{k-p})$ is DSH and its support is in a neighborhood of $E$. Indeed, we have
%\begin{align*}
%dd^c(\chi u\omega^{k-p})=dd^c\chi \wedge (u\omega^{k-p})+\chi dd^c u\wedge \omega^{k-p}+d\chi\wedge d^cu\wedge \omega^{k-p}.
%\end{align*}
%Since $dd^c(\chi u\omega^{k-p})$ is a real current and $\supp d\chi, \supp d^c\chi, \supp dd^c\chi$ do not intersect $E$, $dd^c(\chi u\omega^{k-p})$ is bounded below by a $-M\omega^{k-p}$ for some $M>0$. Hence, from our hypothesis, we have
%\begin{align*}
%\langle U_S, dd^c(\chi u\omega^{k-s})\rangle=\langle S-\omega^s, \chi u\omega^{k-s}\rangle
%\end{align*}
%is finite. Since $u$ is integrable with respect to $\omega^k$, $\langle S, \chi u\omega^{k-s}\rangle$ is finite. This implies that $S$ has no mass on $I_-$.
%\end{proof}

\begin{proposition}\label{prop:induction} Assume that $f$ is as above. 
\begin{enumerate}
\item If $R\in\cC_{k-s+1}$ is smooth near $I_+$, then $f_*R$ is well-defined.
\item Every $S_n$ is well-defined and belongs to $C_s$.
\item For every $n\geq 0$, $f^*S_n$ is well-defined.
%\item $S_n=d_+^{-sn}(f^n)^*S$.
%\item $\cU_{S_n}(R)=\cU_{d_+^{-sn}(f^n)^*S}(R)$ for every $R\in\cC_{k-s+1}(\PP^k\setminus I_+)$.
\end{enumerate}
\end{proposition}

Now, we prove the second main theorem of the paper. For simplicity, we assume that all the super-potentials be of mean $0$ in the rest of the paper.
%\begin{customthm}{\ref{thm:2ndmainthm}}[cf. Remark 5.5.8 in \cite{DS2009}]
%Let $f:\CC^k\to\CC^k$ be a regular polynomial automorphism of $\CC^k$ of degree $d_+\geq 2$ and $p>0$ an integer such that $\dim I_+=k-p-1$ and $\dim I_-=p-1$ where $I_\pm$ are the set of indeterminancy of $f^{\pm 1}$, repsectively. Then, for $S\in\cC_p$ whose super-potential $\cU_S$ of mean $0$ is bounded near $I_-$, $d_+^{-pn}(f^n)^*S$ converges to $T_+^p$ in the sense of currents.
%\end{customthm}

\begin{proof}[Proof of Theorem \ref{thm:2ndmainthm}] 
Let $W_f\Subset\PP^k\setminus \overline{K_+}$ denote a neighborhood of $I_-$ satisfying Definition \ref{def:bounded} with the constant $c_S>0$. Let $R:=R_+-R_-$ be a smooth closed current of bidegree $(k-p+1, k-p+1)$ such that $R_\pm$ are smooth positive closed and of the same mass, and $\supp R\Subset \PP^k\setminus \overline{K_+}$. Note that $\Lambda_p^i(R_\pm)$'s are all smooth outside $I_-$ for any $i\in\NN$. Then, Proposition 5.1.8 in \cite{DS2009} implies 
%\begin{align*}
%|\cU_{L_p(S)}(R)-\cU_{T_+^p}(R)|=d^{-1}|\cU_S(\Lambda_p(R))-\cU_{T_+^p}(\Lambda_p(R))|.
%\end{align*}
%Due to (4) of Proposition \ref{prop:polyautobasic}, we can apply this inequality inductively and we have
\begin{align*}
|\cU_{L_p^n(S)}(R)-\cU_{T_+^p}(R)|= d^{-n}|\cU_S(\Lambda_p^n(R))|+d^{-n}|\cU_{T_+^p}(\Lambda_p^n(R))|.
\end{align*}
We have $\|\Lambda_p^n(R_\pm)\|=\|R_\pm\|$ for all $n$. Due to (4) of Proposition \ref{prop:polyautobasic}, since $\cU_S$ is bounded in $W_f$, for all sufficiently large $n$, we have 
\begin{displaymath}
d^{-n}|\cU_S(\Lambda_p^n(R))|\leq d^{-n}c_S\|R_+\|.
\end{displaymath}

We estimate the second term. Note that $T_+^p$ is not PC any more here. 
%Note that since $\cU_{T_+^p}(\Lambda_p^n(R))=\cU_{T_+^p}(\Lambda_p^n(R_+))-\cU_{T_+^p}(\Lambda_p^n(R_-))$, we do not have to care about the mean of the super-potential. 
We have
\begin{align*}
\cU_{T_+^p}(\Lambda_p^n(R))=\int_{z\in\PP^k}\int_{w\in\PP^k\setminus\{z\}}T_+^p(w)\wedge K(w, z)\wedge \Lambda_p^n(R)(z).
\end{align*}
Due to $\supp T_+^p\in \overline{K_+}$, $\supp R\Subset \PP^k\setminus \overline{K_+}$, and (4) of Proposition \ref{prop:polyautobasic}, we see that $\cU_{T_+^p}(\Lambda_p^n(R))$ is bounded independently of $n$. Hence, we have
\begin{align}\label{eqn:equi_poly}
|\cU_{L_p^n(S)}(R)-\cU_{T_+^p}(R)|\lesssim d^{-n}(1+\|R_+\|).
\end{align}
Let $\varphi$ be a test smooth $(k-p, k-p)$-form with $\supp \varphi\Subset \PP^k\setminus \overline{K_+}$. Then, we have
\begin{align*}
|\langle L_p^n(S)-T_+^p, \varphi\rangle|=|\langle dd^c(U_{L_p^n(S)}-U_{T_+^p}), \varphi\rangle|=|\langle (U_{L_p^n(S)}-U_{T_+^p}), dd^c\varphi\rangle|.
\end{align*}
Since $\supp \varphi\Subset \PP^k\setminus \overline{K_+}$, \eqref{eqn:equi_poly} implies that $|\langle L_p^n(S)-T_+^p, \varphi\rangle|$ converges to $0$. Hence, %$\supp T_+^p\subseteq \overline{K_+}$ and therefore, 
we see that for any limit point $S'$ of $\{L_p^n(S)\}$ in $\cC_p$, $\langle S', \varphi\rangle=0$ and Theorem 5.5.4 in \cite{DS2009} implies the desired equidistribution.
\end{proof}

%In some cases, we can estimate the speed of convergence.
Next, we prove Theorem \ref{thm:2ndmainthmspeed}.
Notice that it is not clear whether $\cU_S(T_+^{k-s+1})$ is finite. So, we cannot directly apply the proof of Theorem \ref{thm:mainthm}. 
%Instead, as a normalizing point in $\cC_{k-s+1}$ for the dynamical super-potentials, we use $R_\infty:=(\deg I_-)^{-1}T_+^{p-s}\wedge[I_-]$. This is well-defined due to the work of Bedford-Taylor in \cite{BT}. 
The following is from direct computations.

\begin{proposition}
We have $\Lambda_s(R_\infty)=R_\infty$.
\end{proposition}
%
%\begin{proof}
%The current $T_+$ admits quasi-potentials H\"older continuous outside $I_+$. Be aware that $f$ is proper over $I_-$.
%\begin{align*}
%&f_*(T_+^{p-s}\wedge[I_-])=d_+^{s-p}f_*((f^*(T_+))^{p-s}\wedge[I_-])\\
%&\quad=d_+^{s-p}f_*(\lim_{\theta\to 0}(f^*(T_+_\theta))^{p-s}\wedge[I_-])=\lim_{\theta\to 0}d_+^{s-p}f_*((f^*(T_+_\theta))^{p-s}\wedge[I_-])\\
%&\quad=\lim_{\theta\to 0}d_+^{s-p}f_*(f^*((T_+_\theta)^{p-s})\wedge[I_-])=\lim_{\theta\to 0}d_+^{s-p}(T_+_\theta)^{p-s}\wedge f_*[I_-]\\
%&\quad=\lim_{\theta\to 0}d_+^{s-1}(T_+_\theta)^{p-s}\wedge [I_-]=d_+^{s-1}T_+^{p-s}\wedge [I_-].
%\end{align*}
%\end{proof}

Similarly to Section \ref{sec:proof1}, we define
\begin{displaymath}
\cV_S:=\cU_S-\cU_{T_+^s}-c_S \textrm{ and }V_S:=U_S-U_{T_+^s}-(m_S-m_{T_+^s}+c_S)\omega^{s-1}
\end{displaymath}
where $\cU_S$ and $U_S$ (resp., $\cU_{T_+^s}$ and $U_{T_+^s}$) are the super-potential of mean $0$ and the Green quasi-potential of $S$ (resp., $T_+^s$), $m_S$ (resp., $m_{T_+^s}$) is the mean of $U_S$ (resp., $U_{T_+^s}$) and $c_S:=\cU_S(R_\infty)-\cU_{T_+^s}(R_\infty)$.
Notice that $\cU_S(R_\infty)$ and $\cU_{T_+^s}(R_\infty)$ are finite.

\begin{lemma}[Lemma 5.5.5 in \cite{DS2009}]\label{lem:exponential}
Let $W\Subset \PP^k\setminus I_+$ be an open set. Then,
\begin{enumerate}
\item $\cV_S(R_\infty)=0$,
\item $\cV_S(R)=\langle V_S, R\rangle$ for smooth $R$ in $\cC_{k-s+1}(W)$,
\item $\cV_{L_s(S)}=d_+^{-1}\cV_S\circ\Lambda_s$ on $\cC_{k-s+1}(W)$ and
\item $\cU_S-\cV_S$ is bounded on $\cC_{k-s+1}(W)$ by a constant independent of $S$.
\end{enumerate}
%Here, $\cC_{k-s+1}(W)$ is the space of currents in $\cC_{k-s+1}$ whose support is compact in $W$.
Here, $\cC_{k-s+1}(W):=\{\Psi\in\cC_{k-s+1}: \supp \Psi\Subset W\}$.
\end{lemma}

%Let $f:\PP^k\setminus I_+\to \PP^k$. Then, its critical set is $I_-$. Let $\epsilon>0$ be sufficiently small. Let $s_{n,i}=\epsilon^{ni}$. Definition of $R_{n, i}$.

\begin{proof}[Proof of Theorem \ref{thm:2ndmainthmspeed}] Let $\varphi\in\cD^{k-s}(W)$. Then, we have
\begin{align*}
\langle S_n-T_+^s, \varphi \rangle=\langle dd^cV_{S_n}, \varphi \rangle=\langle V_{S_n}, dd^c\varphi \rangle.
\end{align*}
By the same argument in Proposition 3.1.9 in \cite{DS2009}, we can find two smooth positive closed $(k-s+1, k-s+1)$-currents $\Phi_\pm$ such that $\Phi_+-\Phi_-=dd^c\varphi$ and $\supp \Phi_\pm\Subset \widetilde{W}$ for some $\widetilde{W}\Subset \PP^k\setminus I_+$. We will replace $W$ by $\widetilde{W}$. Then, it suffices to show that $\cV_{S_n}(R)\to 0$ exponentially fast as $n\to\infty$ where $R$ is a smooth current in $\cC_{k-s+1}(W)$. Due to Lemma \ref{lem:exponential}, we have
\begin{align*}
\cV_{S_n}(R)=d_+^{-n}\cV_{S}(\Lambda_s^n(R))=\lim_{\theta\to 0}d_+^{-n}\cV_{S_\theta}(\Lambda_s^n(R)).
\end{align*}
As in the proof of Theorem \ref{thm:mainthm}, it suffices to estimate $d_+^{-n}\cV_{S_\theta}(\Lambda_s^n(R))$ from below.
\smallskip

%Let $c>0$ be a constant such that $|\cU_S(dd^c(\chi U_{R_0}))|<c$ for every $R_0\in\cC_{k-s+1}$ smooth outside $W_1$. 
For our $R$, we take $\theta\in\CC$ with sufficiently small $|\theta|\ll 1$ so that $|\cU_{S_\theta}(dd^c(\chi U_R))|$ $<c+1$ where $\cU_{S_\theta}$ is the super-potential of $S_\theta$ of mean $0$. Fix an open neighborhood $W'$ of $I_-$ such that $f(W')\Subset W'$ and $W'\Subset W_1\Subset W_2$.

The current $\Lambda_s^n(R)$ is smooth outside $I_-$. Since $\PP^k\setminus W'$ contains $\overline{K_+}$ and is disjoint from $I_-$, we have $f^{-1}(\PP^k\setminus W')\subset \PP^k\setminus W'$. Let $M>0$ be an upper bound of $\|Df^{-1}\|_{pt}$ over $\PP^k\setminus W'$. It follows that $\|\Lambda_s^n(R)\|_{\infty, \PP^k\setminus W'}\lesssim M^{3kn}$. 
%If $M<d_+$, then we are done. So, we assume that $M\geq d_+$. 
We take $\varepsilon:=\max\{d_+, M\}^{-3kn}$. Then, since $S_\theta$ is smooth, we have
\begin{align*}
&\cU_{S_\theta}(\Lambda_s^n(R))=\langle U_{S_\theta}, \Lambda_s^n(R) -(\Lambda_s^n(R))_\varepsilon\rangle_{W'}+\langle U_{S_\theta}, \Lambda_s^n(R) -(\Lambda_s^n(R))_\varepsilon\rangle_{\PP^k\setminus W'}\\
&\quad\quad + \cU_{S_\theta}((\Lambda_s^n(R))_\varepsilon)-m_{S_\theta}\\
&\quad\geq \langle U_{S_\theta}, \Lambda_s^n(R)\rangle_{W'}+\langle U_{S_\theta}, \Lambda_s^n(R) -(\Lambda_s^n(R))_\varepsilon\rangle_{\PP^k\setminus W'} + \cU_{S_\theta}((\Lambda_s^n(R))_\varepsilon)-m_{S_\theta}
\end{align*}
where $U_{S_\theta}$ is the Green quasi-potential of $S_\theta$ and $m_{S_\theta}$ is the mean of $U_{S_\theta}$. Note that $m_{S_\theta}$ is bounded independently of $\theta$ due to the remark 2.3.4 in \cite{DS2009}.

We consider $\langle U_{S_\theta}, \Lambda_s^n(R)\rangle_{W'}$. 
\begin{align*}
\langle U_{S_\theta}, \Lambda_s^n(R)\rangle_{W'}&=\int_{z\in W'}\int_{w\in W_1}S_\theta(w)\wedge K(z, w)\wedge \Lambda_s^n(R)(z)\\
&\quad+\int_{z\in W'}\int_{w\in \PP^k\setminus W_1}S_\theta(w)\wedge K(z, w)\wedge \Lambda_s^n(R)(z).
\end{align*}
Notice that the second integral is bounded independently of $\theta$ and $n$. Let $m_{\Lambda_s^n(R)}$ denote the mass of the Green quasi-potential of $\Lambda_s^n(R)$. By use of the regularization of currents, the first integral can be estimated by
\begin{align*}
&\int_{z\in W'}\int_{w\in W_1}S_\theta(w)\wedge K(z, w)\wedge \Lambda_s^n(R)(z)\geq \int _{W_1}S_\theta\wedge U_{\Lambda_s^n(R)}\\
&\quad\geq \int S_\theta \wedge (\chi U_{\Lambda_s^n(R)})=\int \omega^s \wedge (\chi U_{\Lambda_s^n(R)})+\int U_{S_\theta} \wedge dd^c(\chi U_{\Lambda_s^n(R)})\\
&\quad\geq m_{\Lambda_s^n(R)}-(c+1).
\end{align*}
Notice that the number $m_{\Lambda_s^n(R)}-(c+1)$ is bounded independently of $\theta$ and $n$.
%As previously, Remark 2.3.4 in \cite{DS2009}, this is uniformly bounded independently of $\theta$ and $n$.

From Lemma 2.1.8 in \cite{DS2009} and the discussion in the above, we have
\begin{align*}
|\langle U_{S_\theta}, \Lambda_s^n(R) -(\Lambda_s^n(R))_\varepsilon\rangle_{\PP^k\setminus W'}|\lesssim \varepsilon M^{3kn}\leq 1.
\end{align*}
Hence, it is uniformly bounded independently of $\theta$ and $n$.

Since $S_\theta\in\cC_s$ and $(\Lambda_s^n(R))_\varepsilon\in\cC_{k-s+1}$, Lemma 3.2.10 in \cite{DS2009} implies
\begin{align*}
|\cU_{S_\theta}((\Lambda_s^n(R))_\varepsilon)|\lesssim |\log \varepsilon|\approx n.
\end{align*}

%All the inequalities and estimates in the above are independent of $\theta\in\CC$ with sufficiently small $|\theta|\ll 1$. 

By Theorem 3.1.2 in \cite{DS2009}, $\cU_{T_+^s}$ is bounded above over $\cC_{k-s+1}(W)$ and the mean of the Green quasi-potential is bounded. Since all the estimates in the above are independent of $\theta\in\CC$ with sufficiently small $|\theta|\ll 1$, $\cV_{S_\theta}(\Lambda_s^n(R))\gtrsim -n$ where the inequality is independent of $\theta$ and $n$. Hence, we have just proved the exponential convergence over $\cC_{k-s+1}(W)$ and in particular, we have $\liminf_{n\to\infty}\cV_{S_n}(R)\geq 0$.
\end{proof}

%% References

%\bibliographystyle{alpha}
%\bibliography{refs}

\begin{thebibliography}{FLMn83}

\bibitem[Ahn16]{Ahn2015}
Taeyong Ahn.
\newblock Equidistribution in higher codimension for holomorphic endomorphisms
  of $\mathbb{P}^k$.
\newblock {\em Trans. Amer. Math. Soc.}, 368:3359--3388, 2016.

\bibitem[BD01]{BD}
Jean-Yves Briend and Julien Duval.
\newblock Deux caract\'{e}risations de la measure d'\'equilibre d'un
  endomorphisme de $\mathbb{P}^k(\mathbb{C})$.
\newblock {\em Publ. Math. Inst. Hautes \'{E}tudes Sci.}, 93:145--159, 2001.

\bibitem[Dem]{Demailly}
Jean-Pierre Demailly.
\newblock Complex analytic and differential geometry.
\newblock available at www.fourier.ujf-grenoble.fr/$\sim$demailly.

\bibitem[Din09]{D2009}
Tien-Cuong Dinh.
\newblock Analytic multiplicative cocycles over holomorphic dynamical systems.
\newblock {\em Complex Var. Elliptic Equ.}, 54:243--251, 2009.

\bibitem[DS05]{DS2004}
Tien-Cuong Dinh and Nessim Sibony.
\newblock Green currents for holomorphic automorphisms of compact K\"ahler
  manifolds.
\newblock {\em J. Amer. Math. Soc.}, 18(2):291--312, 2005.

\bibitem[DS06]{DS2006}
Tien-Cuong Dinh and Nessim Sibony.
\newblock Distribution des valeurs de transformations m\'eromorphes et
  applications.
\newblock {\em Comment. Math. Helv.}, 81:221--258, 2006.

\bibitem[DS08]{DS2008}
Tien-Cuong Dinh and Nessim Sibony.
\newblock Equidistribution towards the Green current for holomorphic maps.
\newblock {\em Ann. Sci. \'Ec. Norm. Sup\'er.}, 41:307--336, 2008.

\bibitem[DS09]{DS2009}
Tien-Cuong Dinh and Nessim Sibony.
\newblock Super-potentials of positive closed currents, intersection theory and
  dynamics.
\newblock {\em Acta Math.}, 203:1--82, 2009.

\bibitem[DS10a]{DS2010}
Tien-Cuong Dinh and Nessim Sibony.
\newblock Equidistribution speed for endomorphisms of proejctive spaces.
\newblock {\em Math. Ann.}, 347, 2010.

\bibitem[DS10b]{DS2010-1}
Tien-Cuong Dinh and Nessim Sibony.
\newblock Super-potentials for currents on compact K\"ahler manifolds and
  dynamics of automorphisms.
\newblock {\em J. Algebraic Geom.}, 19:473--529, 2010.

\bibitem[Fav00a]{F2000}
Charles Favre.
\newblock {\em Dynamique de applications rationelles}.
\newblock PhD thesis, Universit\'e de Paris-Sud, Orsay, 2000.

\bibitem[Fav00b]{F2000-1}
Charles Favre.
\newblock Multiplicity of holomorphic functions.
\newblock {\em Math. Ann.}, 316:355--378, 2000.

\bibitem[FJ03]{brolinthm}
Charles Favre and Mattias Jonsson.
\newblock {B}rolin's theorem for curves in two complex dimensions.
\newblock {\em Ann. Inst. Fourier (Grenoble)}, 53:1461--1501, 2003.

\bibitem[FJ07]{eigenvaluations}
Charles Favre and Mattias Jonsson.
\newblock Eigenvaluations.
\newblock {\em Ann. Sci. \'{E}cole Norm. Sup.}, 40:309--349, 2007.

\bibitem[FLMn83]{FLM}
Alexandre Freire, Artur Lopes, and Ricardo Ma\~{n}\'{e}.
\newblock An invariant measure for rational maps.
\newblock {\em Bol. Soc. Brasil. Mat.}, 14:45--62, 1983.

\bibitem[FS94]{FS1994}
John~Erik Forn{\ae}ss and Nessim Sibony.
\newblock {\em Complex potential theory (Montreal, PQ, 1993)}, volume 439 of
  {\em NATO Adv. Sci. Inst. Ser. C Math. Phys. Sci.}, chapter Complex dynamics
  in higher dimensions, pages 131--186.
\newblock Kluwer Acad. Publ., 1994.

\bibitem[FS95]{FS1995}
John~Erik Forn{\ae}ss and Nessim Sibony.
\newblock {\em Modern methods in complex analysis (Princeton, NJ, 1992)},
  volume 137 of {\em Ann. of Math. Stud.}, chapter Complex dynamics in higher
  dimension. {II}, pages 135--182.
\newblock Princeton Univ. Press, 1995.

\bibitem[GH94]{GH}
Phillip Griffiths and Joseph Harris.
\newblock {\em Principles of Algebraic Geometry}.
\newblock Wiley Classics Library. John Wiley $\&$ Sons, Inc., New York, reprint
  of the 1978 original edition, 1994.
\newblock xiv+813 pp.

\bibitem[Gig14]{G2014}
William Gignac.
\newblock Measures and dynamics on noetherian spaces.
\newblock {\em J. Geom. Anal.}, 24:1770--1793, 2014.

\bibitem[Gri78]{Griffiths}
Phillip Griffiths.
\newblock Complex differential and integral geometry and curvature integrals
  associated to singularities of complex analytic varieties.
\newblock {\em Duke Math. J.}, 45:427--512, 1978.

\bibitem[Gue03]{Guedj}
Vincent Guedj.
\newblock Equidistribution towards the {G}reen current.
\newblock {\em Bull. Soc. Math. France}, 131:359--372, 2003.

\bibitem[HP75]{HarveyPolking}
Reese Harvey and John Polking.
\newblock Extending analytic objects.
\newblock {\em Comm. Pure Appl. Math.}, 28:701--727, 1975.

\bibitem[Lju83]{Lyubich}
Mikhail Ljubich.
\newblock Entropy properties of rational endomorphisms of the {R}iemann sphere.
\newblock {\em Ergodic Theory Dynam. Systems}, 3:351--385, 1983.

\bibitem[Mal67]{Mal}
Bernard Malgrange.
\newblock {\em Ideals of differentiable functions}.
\newblock Oxford University Press, 1967.

\bibitem[Par11]{Parra}
Rodrigo Parra.
\newblock {\em Currents and equidistribution in holomorphic dynamics}.
\newblock PhD thesis, Univ. of Michigan, 2011.

\bibitem[RS97]{RS}
Alexander Russakovskii and Bernard Shiffman.
\newblock Value distribution for sequences of rational mappings and complex
  dynamics.
\newblock {\em Indiana Univ. Math. J.}, 46:897--932, 1997.

\bibitem[Sib99]{sibony}
Nessim Sibony.
\newblock Dynamique des applications rationnelles de $\mathbb{P}^k$.
\newblock {\em Panoramas $\&$ Synth\`{e}ses}, 8:97--185, 1999.

\bibitem[Taf11]{Taflin2011}
Johan Taflin.
\newblock Equidistribution speed towards the Green current for endomorphisms of
  $\mathbb{P}^k$.
\newblock {\em Adv. Math.}, 227:2059--2081, 2011.

\end{thebibliography}

\end{document}